\documentclass[11pt,letter]{article}

\usepackage[margin=1in]{geometry}

\usepackage{dsfont, amssymb,amsmath,amscd,latexsym, amsthm, amsxtra,amsfonts}

\usepackage[dvipsnames,svgnames]{xcolor}

\usepackage{lineno}

\usepackage[active]{srcltx}

\usepackage{tikz}

\usepackage{natbib}
 \bibpunct[, ]{(}{)}{,}{a}{}{,}%

\usepackage{bbm}
\usepackage{enumerate}
\usepackage{mathrsfs}

\usepackage{comment}         
\usepackage{cases}
\usepackage{circuitikz}

\usepackage{subcaption} 

\usepackage{verbatim}
\usepackage{graphicx}
\usepackage{epstopdf}
\usepackage{bm}
\usetikzlibrary{automata,arrows,positioning,calc}
\usepackage[colorlinks=true,citecolor=db,linkcolor=db,urlcolor=blue,pdfstartview=FitH]{hyperref}
\numberwithin{equation}{section}
\usepackage{graphics}
\usepackage{booktabs}
\usepackage{algorithm,algorithmic}

\usepackage{mathtools}
\mathtoolsset{showonlyrefs=true}

\usepackage{adjustbox}                  

\usetikzlibrary{arrows.meta, positioning, calc, shadows.blur}

\definecolor{brandBlue}{HTML}{0072B2}
\definecolor{brandOrange}{HTML}{E69F00}
\definecolor{brandGreen}{HTML}{009E73}
\definecolor{brandPurple}{HTML}{6F42C1}
\definecolor{db}{RGB}{0, 0, 130}
\definecolor{rp}{rgb}{0.25, 0, 0.75}
\definecolor{dg}{rgb}{0, 0.6, 0}

\allowdisplaybreaks

\newtheorem{theorem}{Theorem}[section]
\newtheorem*{theorem*}{Theorem}
\newtheorem{corollary}[theorem]{Corollary}
\newtheorem{proposition}[theorem]{Proposition}
\newtheorem{lemma}[theorem]{Lemma}

\theoremstyle{definition}
\newtheorem{definition}[theorem]{Definition}

\theoremstyle{definition}

\theoremstyle{definition}
\newtheorem{remark}{Remark}
\theoremstyle{definition}
\newtheorem{assumption}{Assumption}

\usepackage{amsmath}
\usepackage{stmaryrd}

\newcommand{\mR}{\mathbb{R}}

\newcommand{\mN}{\mathbb{N}}

\newcommand{\mE}{\mathbb{E}}
\newcommand{\mF}{\mathbb{F}}
\newcommand{\mP}{\mathbb{P}}

\newcommand{\mL}{\mathbb{L}}
\newcommand{\mB}{\mathbb{B}}

\renewcommand{\epsilon}{\varepsilon}

\newcommand{\M}{\mathcal{M}}

\newcommand{\F}{\mathcal{F}}

\newcommand{\Ti}{\mathcal{T}}

\newcommand{\B}{\mathcal{B}}

\newcommand{\cR}{\mathcal{R}}
\newcommand{\cX}{\mathcal{X}}

\newcommand{\cA}{\mathcal{A}}
\newcommand{\cW}{\mathcal{W}}
\newcommand{\cL}{\mathcal{L}}

\newcommand{\cP}{\mathcal{P}}

\newcommand{\barX}{\bar{X}}

\newcommand{\bP}{\mathbf{P}}
\newcommand{\bJ}{\mathbf{J}}

\newcommand{\bB}{\mathbf{B}}
\newcommand{\fB}{\mathfrak{B}}

\newcommand{\bmu}{\bm{\mu}}
\newcommand{\bnu}{\bm{\nu}}

\newcommand{\sC}{\mathscr{C}}

\newcommand{\sD}{\mathscr{D}}

\newcommand{\bD}{\mathbf{D}}

\newcommand{\mt}{\mathbf{t}}

\newcommand{\Tr}{\mathrm{Tr}}
\newcommand{\mc}{\mathrm{c}}

\newcommand{\cQ}{\mathcal{Q}}
\newcommand{\dtriangle}{\delta \triangle}

\newcommand{\llb}{\llbracket}
\newcommand{\rrb}{\rrbracket}
\newcommand{\bw}{\mathbf{w}}
\newcommand{\mb}{\mathrm{b}}

\newcommand{\fC}{\mathfrak{C}}

\title{Mean-field games with rough common noise:\\ the compactification approach}

\author{
Erhan Bayraktar \thanks{Department of Mathematics, University of Michigan, Ann Arbor, MI, USA. 
		\url{erhan@umich.edu}. Supported in part by the National Science Foundation under grant DMS-2507940 and by the
Susan M. Smith Chair.}
\and Xihao He \thanks{Department of Mathematics, University of Michigan, Ann Arbor, MI, USA. 
		\url{hexihao@umich.edu}}
\and Xiang Yu \thanks{Department of Applied Mathematics, the Hong Kong Polytechnic University, Kowloon, Hong Kong. 
		\url{xiang.yu@polyu.edu.hk}.  Supported by the Hong Kong RGC General Research Fund (GRF) under grant no.15214125.  }
\and Fengyi Yuan \thanks{School of Science and Engineering, The Chinese University of Hong Kong (Shenzhen),
Shenzhen, Guangdong, China. 
		\url{yuanfengyi@cuhk.edu.cn}. Supported by the Chinese University of Hong Kong (Shenzhen) start-up fund under UDF01004253.}
}

\date{}

\begin{document}
\maketitle 

\begin{abstract}
     We study mean-field game (MFG) problems with rough common noise, in which the representative state dynamics are governed by a controlled rough stochastic differential equation driven by an idiosyncratic Brownian motion and a deterministic rough-path signal that affects the whole population. Within this new framework, we introduce a canonical weak formulation based on relaxed controls and rough martingale problems. We prove the existence of a pathwise mean-field equilibrium by developing new compactification tools that accommodate rough integration and differ substantially from classical compactification arguments in the literature. Finally, we discuss the relationship between the pathwise problem and the classical MFG problem with randomized Brownian common noise. { Using the notion of a pathwise admissible set, we recast mean-field game problems with common noise as optimization problems over an extended space of probability measures. We establish an equivalent characterization of Carmona-Delarue-Lacker's weak equilibrium and, as an application, give an alternative proof of strong equilibrium without first establishing pathwise uniqueness}.


\end{abstract}
\bigskip
 
\section{Introduction}

The primary goal of this paper is to study a mean-field game (MFG) problem with rough common noise, which is formulated as follows:
a representative agent solves an optimal control problem under the given population aggregation $\bmu :=\{\mu_t\}_{t \in [0,T]}$, where the state process $X^{\bmu,\alpha}$ solves a controlled rough stochastic differential equation (RSDE):
\begin{align}\label{state-dynamic-1}
    dX^{\bmu,\alpha}_t  = b(t,X^{\bmu,\alpha}_t,\mu_t,\alpha_t) dt
                            ~ + ~
                            \sigma(t,X^{\bmu,\alpha}_t,\mu_t)dW_t
                            ~ + ~
                            \sigma^0(t,X^{\bmu,\alpha}_t,\mu_t)d\bB_t.
\end{align}
Here $\alpha := \{\alpha_t\}_{t\in [0,T]}$ is an admissible control, and the coefficient fields $b,\sigma,\sigma^0$ are given and satisfy technical assumptions stated below. The objective functional of the representative agent is
\begin{align*}
    J(\alpha;\bmu) := \mE\bigg[\int_0^T f(t,X^{\bmu,\alpha}_t,\mu_t,\alpha_t)dt
    + g(X^{\bmu,\alpha}_T,\mu_T)\bigg],
\end{align*}
where the running cost function $f$ and the terminal cost function $g$ are given. As in classical MFG, the mathematical problem is to find a Nash equilibrium $(\alpha_*,\bmu_*)$, where $\alpha_*$ is the best response for the representative agent and $\bmu$ precisely characterizes the environment when all agents use the strategy $\alpha_*$. The main novelty of this paper is the introduction of two distinct processes to model two types of noise. On the one hand, each agent in the population has its own idiosyncratic noise path, which is modeled probabilistically by the Brownian motion $W$. On the other hand, the common noise is shared by the entire population, and we use a {\it single} rough path $\bB$ to represent such a common shock, without imposing a probabilistic model.

Dating back to \cite{McKean1966NonlinearParabolic}, mean field approximation originated in statistical physics, where many-body interactions were replaced by their average effects to make models tractable. When this idea was transferred from passive particles to strategic agents, mean field games (MFGs) were introduced and studied by \cite{LasryLions2007} and \cite{HuangCainesMalhame2006}. They provide a tractable framework for analyzing strategic interactions in large populations of nearly symmetric agents. Beyond their intrinsic mathematical interest, MFGs have become standard models in economics, finance, energy systems, and engineering, where the number of interacting decision makers is large and the aggregate effect of the crowd is essential.

Practical applications with exogenous shocks affecting the entire system have motivated the introduction of common noise in mean field problems. Depending on the context, common noise may represent macroeconomic changes, sudden demand shifts, policy or regulatory actions, or environmental events. We
refer to \cite{carmona2018probabilistic2} for a comprehensive introduction to recent developments on this topic. Despite its importance, common noise creates structural differences from classical settings. The main difficulty lies in replacing a deterministic measure flow $\mu_t=\cL(X_t)$ by a flow of {\it conditional} laws $\mu_t=\cL(X_t|\F^B_t)$ when the common noise is modeled by a Brownian motion $B$. In a specific linear model, \cite{Ahuja2016WellposedMFGCommonNoise} establishes the existence of mean-field equilibrium under a monotonicity condition. In the analytical approach, which originates in \cite{LasryLions2007}, mean-field games are studied either through a coupled system of an HJB equation and a Fokker-Planck equation or through an infinite-dimensional master equation. In the case of common noise, these analytical objects become {\it stochastic} partial differential equations, requiring much more involved technical analysis than in models without common noise. \cite{CardaliaguetDelarueLasryLions2019MasterEquation} establish the well-posedness of the stochastic master equation, and \cite{CardaliaguetSouganidis2022FirstOrderMFGCommonNoise} provides similar results when the idiosyncratic noise is degenerate. The analytical approach usually relies on structural assumptions, such as monotonicity conditions, convexity of Hamiltonians, or constant volatility coefficients. A pioneering work in a more general setting is \cite{Lacker-common-noise}, where the authors use a probabilistic compactification approach to study a weakened concept of mean-field equilibrium allowing $\mu_t=\cL(X_t|\F^{B,\mu}_t)$. This weaker definition of equilibrium seems indispensable for resolving the topological issues of conditional laws; namely, the best-response map composed with the conditional-law map is not continuous under the naive topology of almost sure weak convergence. Later, this concept of ``weak equilibrium" also motivated a similar notion of weak solution to McKean-Vlasov dynamics with common noise in \cite{HammersleySiskaSzpruch2021WeakMKVCommonNoise}.

A well-known drawback of weak equilibrium is its potential dependence on additional information beyond the common noise. In reality, $\bmu=\{\mu_t\}_{t\in [0,T]}$ represents the statistical distribution of all agents and should therefore be completely determined once the common noise, the sole source of uncertainty at the systematic level, is realized. This motivates us to adopt a pathwise perspective on common noise: instead of conditioning on a random input, we fix it as a deterministic path within the state dynamics; see the $d \bB$-term in \eqref{state-dynamic-1}. The pathwise perspective on state equations and control problems has been proposed and investigated in several models, yielding equivalent formulations by freezing the external randomness in the original problem; see, for instance, \cite{pathwisecontrol}, \cite{BailleulCatellierDelarue_MFRDE}, \cite{CoghiDeuschelFrizMaurelli2020}, \cite{controleedRSDE}, \cite{BoWangWeiYu_PathwiseCompact} among others. In our pathwise formulation of MFGs, the randomness comes from discrepancies among the idiosyncratic noises of different agents and potentially from external randomization of control (i.e., the relaxed control formulation; see Definition \ref{def:pathwise-relax-control}), but not from the common noise. Consequently, the measure flow $\bmu$ becomes deterministic, eliminating the technical issues posed by conditional laws; see the analogous pathwise treatment of Poissonian common noise in mean-field control problems in the recent study \cite{BoWangWeiYu_PathwiseCompact}.
On the other hand, unlike the Poisson case in \cite{BoWangWeiYu_PathwiseCompact},
to capture the low time regularity of the continuous noisy signal $\bB$, rough-path noise is a natural and suitable choice, which is the main motivation for our study of MFG with rough common noise.

Rough path theory, initiated by Lyons and further developed by numerous authors, offers a robust pathwise interpretation for differential equations driven by irregular signals, without requiring statistical assumptions on those signals. It is well known that defining integration becomes challenging when the integrator \( B \) has very low regularity (for example, when it represents a single realized path of a Brownian motion). In contrast to It\^o's probabilistic approach, the core idea of rough path theory is to enhance the driving path \( B \) with a higher-level path \( \mB \) and to consider a convergent extended Riemann sum against \( \bB:=(B,\mB) \). For references on rough path theory, we refer to \cite{Lyons1998,LyonsQian2002,Gubinelli2004,FrizVictoir2010,rough_path_book}. Although the original rough path theory has been well developed and widely applied in recent decades, it is not sufficient for our setting. Indeed, in \eqref{state-dynamic-1}, the Brownian motion term \( d W_t \) makes the solution \( X \), and consequently the integrand \( \sigma^0(t,X_t,\mu_t) \) in the \( d \bB_t \)-term, a stochastic process. This situation is not covered by standard rough path theory and motivates the use of a recent breakthrough in this direction: the theory of {\it rough stochastic differential equations} (RSDE) in \cite{RSDE}. Following this pioneering paper, several studies have extended or applied this framework to various problems. For instance, \cite{roughPDE-dual} explores parametrized RSDEs, \cite{controleedRSDE} investigates pathwise stochastic control problems, \cite{mkv_rough_common_noise} extends RSDEs to the McKean-Vlasov case (with \cite{RFPK} investigating the corresponding Fokker-Planck equations), and \cite{randomisation_RSDE} focuses on their randomization to classical SDEs with two Brownian motions. Near the completion of this manuscript, we came across a preprint \cite{friz2026meanfieldgamesroughcommon} that explores linear-quadratic mean-field game problems with rough common noise and uses new tools for rough forward-backward SDEs.


The aforementioned series of studies by Friz and coauthors plays a crucial role in this work. However, our methodology differs substantially: we develop new techniques for applying probabilistic compactification arguments. Specifically, while \cite{RSDE} and many subsequent studies primarily focus on strong formulations on a fixed probability space, we work on the canonical space and use a relaxed formulation in the weak sense. The main reason for this shift is that it facilitates compactness. Although the compactification approach has been well developed without common noise and with Brownian common noise (\cite{Lacker2015} and \cite{Lacker-common-noise}), it has not been explored in the context of RSDEs. The present paper fills this gap and advances the compactification approach by addressing technical challenges induced by rough paths.

First, when defining the admissible relaxed control (or equivalently, the state dynamics), we cannot rely directly on the original RSDE in \cite{RSDE} or its extension to the McKean-Vlasov case in \cite{mkv_rough_common_noise}, because of the $\mu$-dependence in $\sigma^0$. Unlike the McKean-Vlasov equation, our state dynamics are those of a standard RSDE in \cite{RSDE}, with coefficients depending on $\bmu$. Therefore, the first step is to define the appropriate domain of $\bmu$ (Definition \ref{def:domain-mu}) and identify the correct controlled vector field to ensure the well-posedness of the state dynamics (Lemma \ref{lemma:construct_cvf}). This also introduces an additional condition in the definition of admissible relaxed controls (Condition 2, Definition \ref{def:pathwise-relax-control}).

Second, we propose and study {\it rough martingale problems} associated with relaxed controls (as outlined in Appendix \ref{app:rough-martingale}), which are new to the literature. We establish their complete equivalence with weak solutions to RSDEs. This is a noteworthy mathematical result in its own right, and it is crucial for obtaining our main result: the existence of mean-field equilibrium. The direction from RSDEs to rough martingale problems is a pivotal step in showing that the admissible set is non-empty. Typically, this step is straightforward in classical models without rough-path noise. However, the presence of rough paths creates substantial differences in the proofs. In principle, the rough path norms needed to verify the definition of relaxed controls (Definition \ref{def:pathwise-relax-control}) are {\it not} law-invariant because they depend on the filtration specified in the strong formulation. By carefully tracking the {\it natural} filtration, we still establish certain invariance properties (Lemma \ref{lemma:isometry}), which ultimately lead to the finite rough path norm condition in Definition \ref{def:pathwise-relax-control}. The converse direction, from rough martingale problems to RSDEs, is essential for ensuring tightness of the relaxed control set, in conjunction with a priori estimates of RSDEs (Lemma \ref{lemma:tight:1}).

Finally, and perhaps most importantly, the proof of the fixed point result for the consistency condition differs substantially from the classical arguments used in the absence of rough paths. The most challenging aspect is proving the invariance of the fixed point map (Subsection \ref{subsec:step-2}). To address this, we propose a domain $\cP_{M,\epsilon}$ that explicitly uses properties of rough path norms. In particular, we leverage the compact embedding from $\beta$-H\"older norms to $\alpha$-H\"older norms when $\beta < \alpha$, as detailed in Theorem \ref{thm:pathwiseMFE}. Additionally, due to the inherent nature of rough path norms, the lower hemicontinuity of the fixed-point map cannot be obtained through a straightforward Gronwall-type estimate, as demonstrated in Remark \ref{rmk:hemicts}. Instead, we resort to weak uniqueness of relaxed controls when the control-noise joint law is fixed (Proposition \ref{lemma:R:uniqueness} and Lemma \ref{lemma:Phi:continuity}). This reliance necessitates the explicit encoding of idiosyncratic noise $W$ within Definition \ref{def:pathwise-relax-control}.

In sum, one of the main results of this paper is Theorem \ref{thm:pathwiseMFE}, which establishes the existence of (pathwise) mean-field equilibrium. This theorem builds upon the extensive technical results as elaborated above.

{
As another contribution to the study of MFGs with common noise, in Section \ref{sec:randomization} we establish connections between the proposed pathwise formulation and the conventional problem with Brownian common noise, which we call the randomized problem. We prove that, given any random input $(\hat\bB,\hat\mu)$, admissible controls for randomized problems are exactly those probability measures whose conditional laws are admissible in the pathwise problems, for almost every realization of $(\hat\bB,\hat\mu)$; see Lemma \ref{lemma:randomization}. This relation enables a new equivalent characterization of weak equilibrium in \cite{Lacker-common-noise}; see Proposition \ref{prop:randomized-weak}. As an application, we prove that every weak equilibrium is strong under additional structural assumptions (Assumption \ref{ass:strong}). Compared with \cite{Lacker-common-noise}, the most important new feature of our result is that it relies neither on pathwise uniqueness of weak equilibrium nor on uniqueness of the best-response map. We also note that the results in Section \ref{sec:randomization} support applications with common noise potentially beyond the Brownian type, or even semimartingales.
}

To summarize, the main contributions of this paper are as follows:
\begin{enumerate}
\item We introduce a new framework for mean-field games with rough-path common noise and develop a compactification approach using relaxed controls on the canonical space. This study motivates numerous future research projects within the same framework, such as limit theory from $n$-player games to the mean-field problem. Furthermore, this framework naturally connects to classical MFG with Brownian common noise.

\item Under mild model assumptions, the existence of mean-field equilibrium for MFG with rough common noise is established.

\item We develop new technical tools and results in the rough path context, including the equivalence between RSDEs and rough martingale problems, as well as a weak uniqueness result for relaxed control problems.

\item {We discuss the relationship between rough-path common noise and conventional Brownian common noise and establish an alternative characterization of weak equilibrium based on the pathwise formulation proposed in this paper. }
\end{enumerate}

\vspace{0.5em}

The rest of the paper is organized as follows. Section \ref{sec:pathwiseproblem} formulates the problem and studies basic properties of the admissible set of relaxed controls. Section \ref{sec:pathwiseMFE} proves the existence of a mean field equilibrium using the Kakutani fixed-point theorem. Section \ref{sec:randomization} establishes connections between pathwise problems with rough-path common noise and standard mean field game problems with Brownian common noise. Finally, Appendix \ref{app:causal-coupling} presents auxiliary results on causal coupling, and Appendix \ref{app:rough-martingale} collects useful properties of rough martingale problems used in the previous sections.

\vspace{0.5em}

\noindent\textbf{Notations}. We list below some notations that will be used frequently throughout the paper:
\begin{itemize}
\item $\cX:=C([0,T];\mR^d)$, the space of state paths. Equip $\cX$ with $d_\infty$ metric and Borel filtration $\mF^X=\{\F^X_t\}_{t\in [0,T]}$.
\item {$\cW = C([0,T];\mR^n)$, the space that supports the idiosyncratic noise. Equip $\cW$ with $d_\infty$ metric and Borel filtration $\mF^W=\{\F^W_t\}_{t\in [0,T]}$.}
\item With a compact action space $U$, $\cQ$ is the set of bounded measures on $[0,T]\times U$, with total mass being $T$ and the first marginal being $dt$. We usually identify $q\in \cQ$ with a path $t\mapsto q_t\in \cP(U)$, given by disintegration $q(dt,du)=q_t(du) dt$. Equip $\cQ$ with corresponding Borel filtration $\mF^\Lambda=\{\F^\Lambda_t\}_{t\in [0,T]}$.
\item { $\Omega^0 = C([0,T];\mR^l)$, the space that supports the common noise. Equip $\Omega^0$ with $d_\infty$ metric and Borel filtration $\mF^{B^0}$. $\mP^0$ is the Wiener measure on $\Omega^0$.}

\item $\Omega:=\cX\times\cQ\times \cW$, and $\bar \Omega: = \Omega\times \Omega^0$. Equip both with corresponding product filtration $\mF=\{\F_t\}_{t\in [0,T]}$ and $\bar \mF=\{\bar \F_t\}_{t\in [0,T]}$.
\item The coordinate map on $\Omega$: $(X,\Lambda,W)$; the coordinate map on $\bar \Omega$: $(\bar X,\bar \Lambda,\bar W,\bar B^0)$.
\item $\triangle=\{(s,t):0\leq s< t\leq T\}$, $\dtriangle=\{(s,u,t):0\leq s<u<t\leq T\}$.
\item For a Euclidean space $E$, a probability measure $\mP\in \cP(\Omega)$ and parameters $1\leq m\leq n\leq \infty$, $\kappa\in (0,1]$, the space $C_2^{\kappa} L_{\mP}^{m,n}(E)$ consists of two-parameter processes $A:\triangle\times \Omega \to E$ such that:
\begin{itemize}
\item $A$ is jointly Borel measurable, and is $\triangle \to L_{\mP}^{m}$ continuous;
\item $\displaystyle
 	\| A\|_{\kappa;m,n}:=\sup_{(s,t)\in \triangle} \frac{\|\|A_{s,t}|\F_s\|_m\|_n}{|t-s|^\kappa}<\infty $.
We usually omit the value space $E$ and write $A\in C^\kappa_2L^{m,n}_{\mP}$, although for different $A$ they may take values in different $E$.
\item The space $C^\kappa L_\mP^{m,n}(E)$ contains all stochastic processes $Y:[0,T]\times \Omega\to E$ such that $t\mapsto Y_t$ is continuous from $[0,T]$ to $L^m_{\mP}$ and $\delta Y\in C^\kappa_2 L^{m,n}_{\mP}$.

\end{itemize}
\item Let $V$ and $W$ be finite-dimensional normed vector spaces, and let\footnote{In this paper, we only use $k=0,1$.} $\gamma=k+\alpha$ with $k\in \mN$, $\alpha\in(0,1]$. The H\"older space of order $\gamma$ is
\[
C^{\gamma}(V;W)
:= \left\{ f\in C(V;W)\;:\; |f|_\gamma<\infty \right\},
\]
where the H\"older norm and semi-norm are respectively given by
\begin{align}
&|f|_\gamma
:= \sum_{i=0}^{k} \sup_{x\in V}|D^{(i)}f(x)|
 + [f]_\gamma,\\
&[f]_\gamma
:= \sup_{\substack{x,y\in V\\ x\neq y}}
\frac{|D^{(k)}f(x)-D^{(k)} f(y)|}{|x-y|^{\alpha}},
\end{align}
where $D^{(k)}f(x)$ is the $k$-th order Fr\'echet derivative and
$|\cdot|$ represents different canonical norms on different finite-dimensional spaces.
\end{itemize}


\section{{MFG with Rough Common Noise}}\label{sec:pathwiseproblem}

\subsection{Problem formulation and assumptions}

Given a (random) measure flow $\bmu = \{\mu_t\}_{t\in [0,T]}$, we consider the following controlled state process $X^{\bmu,\alpha}$ for the representative agent, governed by the RSDE
\begin{align}
    dX^{\bmu, \alpha}_t = &b(t,X^{\bmu, \alpha}_t,\mu_t,\alpha_t)dt+{ \sigma(t,X^{\bmu,\alpha}_t,\mu_t)}dW_t
    +\sigma^0(t,X^{\bmu, \alpha}_t,\mu_t)d \bB_t,
\end{align}
where $\alpha$ is a progressively measurable control process valued in the compact set $U$, $W$ represents the idiosyncratic noise for each individual agent, and $\bB$ plays the role of common noise affecting the entire population. For technical convenience, we focus on the model in which only the drift term is controlled.

Let us consider the cost functional of the representative agent given by
\begin{align}\label{orig-obj}
J(\alpha;\bmu) = \mE\bigg[ \int_0^T f(t,X^{\bmu, \alpha}_t,\mu_t,\alpha_t)dt + g(X^{\bmu,\alpha}_T,\mu_T)\bigg].
\end{align}

The following standard assumption is imposed throughout the paper.
\begin{assumption}\label{ass:}
The functions $(b,f) : [0,T] \times \mR^d \times \cP_2(\mR^d) \times U \longrightarrow \mR^d  \times \mR$, $(\sigma,\sigma_0) : [0,T] \times \mR^d \times \cP_2(\mR^d)  \longrightarrow \mR^{d\times l} \times \mR^{d \times k}$, $g:\mR^d \times \cP_2(\mR^d)  \longrightarrow \mR$ are bounded and satisfy the following properties:
\begin{enumerate}
\item For fixed $t \in [0,T]$, $f(t,\cdot)$, $g$ are continuous.

\item The functions $b,\sigma$ are Lipschitz uniformly in time and control in the sense that, there exists a constant $L$, such that for any $x_1,x_2 \in \mR^d$, $\mu_1,\mu_2 \in \cP_2(\mR^d)$,
\begin{align*}
    |b(t,x_1,\mu_1,a) - b(t,x_2,\mu_2,a)|
    ~ &+ ~
    |\sigma(t,x_1,\mu_1) - \sigma(t,x_2,\mu_2)|
    \le L\big(|x_1 - x_2| + \cW_2(\mu_1,\mu_2)\big).
\end{align*}
\item The function $\sigma^0$ is continuously differentiable with respect to $\mu$, with bounded derivatives. Moreover, there exists a constant $L$, such that for any $0\leq s<t\leq T$, $x,y\in\mR^d$, $\mu_1,\mu_2\in \cP_2(\mR^d)$:
\begin{align}
  |\sigma^0(s,x,\mu_1)-&\sigma^0(t,x,\mu_2)| + |\nabla \sigma^0(s, x,\mu_1)-\nabla \sigma^0(t,x,\mu_2)|\\
  +&|\partial_\mu\sigma^0(s,x,\mu_1)(y)-\partial_\mu\sigma^0(t,x,\mu_2)(y)|
  \leq  L\big(|t-s|+\cW_2(\mu_1,\mu_2)\big).
\end{align}
\item There exists a $\gamma\in (1,2]$, such that for each $t\in[0,T]$ and $\mu\in \cP_2(\mR^d)$, $\sigma^0(t,\cdot,\mu)\in C^\gamma(\mR^d;\mR^{d\times k})$, $\partial_\mu\sigma^0(t,\cdot,\mu)\in C^{\gamma-1}(\mR^d;\mR^{d\times k\times k})$, and
\begin{align}
    \sup_{t\in [0,T],\mu\in \cP_2(\mR^d)}\big\{|\sigma^0(t,\cdot,\mu)|_{\gamma}+|\partial_\mu\sigma^0(t,\cdot,\mu)|_{\gamma-1}\big\}<\infty.
\end{align}
\end{enumerate}
\end{assumption}

\subsection{Definitions and preliminaries}

Throughout this paper, the following set of indices will be used frequently\footnote{Recall that $\gamma$ is the spatial regularity of $\sigma^0$, and $\alpha$ is the path-regularity of common noise, so typically $\gamma\in (1,2]$, and $\alpha \in (1/3,1/2)$.}
\begin{align}
\Pi =\bigg\{(\beta,\beta'):\beta,\beta'\in \bigg(\frac{1}{1+\gamma},\alpha\bigg], \beta'\leq (\gamma-1)\beta\bigg\}.
\end{align}

For convenience, we consider the Polish space of rough paths $\hat\sC^{0,\alpha}$, defined as the $\rho_\alpha$-closure of smoothed rough paths (see Subsection 3.1 of \cite{randomisation_RSDE}) with the special bracket $[\bB]_{s,t}=(t-s)I_{k\times k}$, where
\begin{align}
    \rho_{\alpha}(\bB^1,\bB^2)= |\delta B^1- \delta B^2|_\alpha + |\mB^1-\mB^2|_{2\alpha}, \quad \bB^1=(B^1,\mB^1), \bB^2 = (B^2,\mB^2).
\end{align}
The It\^o-enhanced Brownian motion, as the most important example of common noise in the literature on MFG problems, has this special bracket. We emphasize that including a general bracket is only a notational extension that replaces the Lebesgue integral with Young's integral.

The space $\hat\sC^{0,\alpha}$ will mainly be used in Section \ref{sec:randomization} when we connect the pathwise problem to the problem with (Brownian) common noise. In Sections \ref{sec:pathwiseproblem} and \ref{sec:pathwiseMFE}, we will fix a $\bB=(B,\mB)\in \hat \sC^{0,\alpha}$.

For any path $Z$, we denote by $\delta Z$ its increment process, i.e., $\delta Z_{s,t} = Z_t-Z_s$ for any $(s,t)\in \triangle:=\{(s,t):s\leq t\}$.

\begin{definition}[Definition 3.1 of \cite{RSDE}, stochastic controlled rough path]
	For $\mP\in \cP(\Omega)$, $(Z,Z')$ is said to be a stochastic controlled rough path by $\bB$, with regularity $(\beta,\beta')$ and integrability $(m,n)$, if:
	\begin{enumerate}
	\item $Z$ and $Z'$ are $\mF$-progressively measurable;
	\item $\delta Z\in C^{\beta}_2L^{m,n}_{\mP}$;
	\item $Z'\in C^{\beta'} L^{m,n}_{\mP}$;
	\item With $R^Z_{s,t}:= \delta Z_{s,t}-Z'_s\delta B_{s,t}$, we have $\mE^{\mP}_{\cdot}[R^Z] \in C^{\beta+\beta'}_2 L^{m,n}_{\mP}$.
	\end{enumerate}
We denote by $\bD_\bB^{\beta,\beta'} L^{m,n}_{\mP}$ the set of all stochastic controlled rough paths endowed with the norm
\begin{align}\label{def:norm:scrp}
\| (Z,Z')\|_{\bB;\beta,\beta';m,n}:= \|\delta Z\|_{\beta;m,n}+\|Z'\|_{\beta';m,n}+\|\mE^{\mP}_{\cdot} [R^Z]\|_{\beta+\beta';m,n}.
\end{align}
\end{definition}

\begin{remark}
To ensure that a pair $(Z,Z')$ is a stochastic controlled rough path, one typically requires $Z'$ to take values in $\mR^k\otimes V$ if $Z$ takes values in $V$, for some finite-dimensional Euclidean space $V$. In this way, $Z'\delta B_{s,t}$ takes values in $V$.
\end{remark}

Formulating the martingale problem with relaxed controls in the rough-path setting requires the RSDE solution theory of \cite{RSDE}.
For this, we only need the deterministic controlled vector field: it enables us to build the controlled lift $(\hat{\sigma}_0, \hat{\sigma}')$ of $\sigma_0$ pathwise (depending on $\mu$ and $X$).
With this lift in hand, we can verify that $(X, \hat{\sigma}_0)$ is a stochastic controlled rough path under $P$, as required by the martingale problem.


\begin{definition}[Definition 3.8 of \cite{RSDE}, controlled vector field] A pair
\begin{align}
(f,f'):[0,T]\to C_b^\gamma(\mR^d;\mR^{d\times k})	\times C_b^{\gamma-1}(\mR^d; \mR^{d\times k\times k})
\end{align}
is said to be a controlled vector field by $\bB$, with regularity $(\beta,\beta')$ if:
\begin{enumerate}\label{def:cvf}
\item $(f,f')$ are uniformly bounded in the sense that
\begin{align}
\sup_{t\in [0,T]} \{|f_t|_\gamma + |f'_t|_{\gamma-1}\}<\infty.
\end{align}
\item Denoting
\begin{align}
\llb Z \rrb_{\kappa}: = \sup_{(s,t)\in \triangle:s\neq t}\frac{\sup_{x\in \mR^d}|Z_{s,t}(x)|}{(t-s)^\kappa},
\end{align}
we have
\begin{align}
\llb \delta f\rrb_{\beta}+\llb \delta f'\rrb_{\beta'}+\llb \delta \nabla f\rrb_{\beta'}<\infty.
\end{align}
\item With $R^f_{s,t}: = f_t-f_s-f'_s\delta B_{s,t}$, we have $\llb R^f\rrb_{\beta+\beta'}<\infty$.
\end{enumerate}
Let $\sD^{\beta,\beta'}_\bB C^\gamma_b$ be the set of controlled vector fields with the norm defined by
\begin{align}
\llb(f,f')\rrb_{\bB;\beta,\beta'}:=\llb\delta f\rrb_{\beta}+\llb \delta f'\rrb_{\beta'}+\llb \delta \nabla f\rrb_{\beta'}+\llb R^f\rrb_{\beta+\beta'}+\sup_{t\in [0,T]} \{|f_t|_\gamma + |f'_t|_{\gamma-1}\}.
\end{align}

\end{definition}

Let us also introduce the following space of measure flows in order to construct controlled vector fields, which is needed for the well-posedness of RSDE.

\begin{definition}\label{def:domain-mu}
We denote by $\cL(\bD^{\beta,\beta'}_\bB L^{m,n},\lambda)$ the set of $\bmu\in \cP(\cX)$ that can be represented as the distribution of stochastic controlled rough paths in $\bD^{\beta,\beta'}_\bB L^{m,n}$, {with initial distribution $\lambda$}. More specifically, $\bmu\in \cL(\bD^{\beta.\beta'}_\bB L^{m,n},\lambda)$ if and only if, on some filtered probability space $(\Omega',\F',\mF',\mP')$, $\mu=\cL^{\mP'}(Y)$ for some $(Y,Y')\in \bD^{\beta,\beta'}_\bB L^{m,n}_{\mP'}$, with $\cL^{\mP'}(Y_0) = \lambda$. We also denote by $\mu_t = \cL^{\mP'}(Y_t)$ and call $(Y,Y')$ a {\it representation} of $\bmu$.
\end{definition}

The following lemma provides the construction of controlled vector fields needed for the solvability of RSDE. See a related result in the study of Fokker-Planck type PDE in Proposition 4.10 of \cite{RFPK}.

\begin{lemma}\label{lemma:construct_cvf}
Fix indices $(\beta,\beta')\in \Pi$, $ 2\leq m\leq n\leq  \infty$. For any $\bmu\in \cL(\bD^{\beta,\beta'}_\bB L^{m,n},\lambda)$ with arbitrary $\lambda\in \cP(\mR^d)$, define $\tilde\sigma^0_t(x):= \sigma^0(t,x,\mu_t)$. Then there exists a $\tilde \sigma':[0,T]\to C_b^\gamma(\mR^d;\mR^{d\times k\times k})$ such that $(\tilde\sigma^0,\tilde\sigma')\in \sD ^{\beta,\beta'}_\bB C_b^\gamma$. Moreover, for any representation $(Y,Y')$ of $\bmu$, we have the following estimates:
\begin{align}\label{lemma:construct_cvf:est}
\llb (\tilde\sigma^0,\tilde \sigma') \rrb_{\bB;\beta,\beta'}\lesssim	1+\|(Y,Y')\|_{\bB;\beta,\beta';m,n}+\|(Y,Y')\|^2_{\bB;\beta,\beta';m,n},
\end{align}
	where the ``$\lesssim$" hides constants that do not depend on $(Y,Y')$.
\end{lemma}

\begin{proof}
Suppose $\mu_t = \cL^{\mP'}(Y_t)$ on some filtered probability space $(\Omega',\F',\mF',\mP')$. We will show that $\tilde\sigma'_t(x):=\mE^{\mP'}[\partial_\mu\sigma^0(t,x,\mu_t)(Y_t)\cdot Y_t']$ satisfies the requirements. Indeed, condition 1 of Definition \ref{def:cvf} follows directly from the assumptions on $\sigma^0$. To check condition 2, we derive from the assumptions on $\sigma^0$ that
\begin{align}
\sup_{x\in \mR^d} |\tilde\sigma^0_t(x)-\tilde\sigma^0_s(x)|\leq & \sup_{x} |\sigma^0(t,x,\mu_t)-\sigma^0(s,x,\mu_t)| + \sup_{x}|\sigma^0(s,x,\mu_t)-\sigma^0(s,x,\mu_s)|\\
       \lesssim & |t-s| + \cW_2(\mu_t,\mu_s)\\
       \lesssim & |t-s| + \|Y_t-Y_s\|_{L^2_{\mP'}}	\\
       \lesssim & |t-s|+ \|\delta Y\|_{\beta;m,n}|t-s|^\beta \\
       \lesssim & (1+\|(Y,Y')\|_{\bB;\beta,\beta';m,n})|t-s|^\beta.
\end{align}
A similar estimate of $\nabla \tilde \sigma^0$ holds in view of the relation $\beta'\leq (\gamma-1)\beta$. We thus obtain
\begin{align}\label{lemma:construct:cvf:est1}
\llb \delta \tilde \sigma^0 \rrb_\beta + \llb \delta \nabla \tilde\sigma ^0\rrb_{\beta'}\lesssim 1+\|(Y,Y')\|_{\bB;\beta,\beta';m,n}.
\end{align}
Next, we estimate $\llb \tilde \sigma'\rrb_{\beta'}$. To this end, we utilize the properties of $\partial_\mu \sigma^0$ to get that
\begin{align}
\sup_{x\in \mR^d}|\tilde \sigma'_t(x)-\tilde\sigma'_s(x)|\leq & \sup_x \mE^{\mP'}[|\partial_\mu\sigma^0(t,x,\mu_t)(Y_t)- \partial_\mu \sigma^0(s,x,\mu_s)(Y_s)|\cdot |Y_t'|]\\
&+\sup_x \mE^{\mP'}[|\partial_\mu \sigma^0(s,x,\mu_s)(Y_s)|\cdot|Y'_t-Y'_s|]	\\
\lesssim & \mE^{\mP'}[(|t-s|+\|Y_t-Y_s\|_{L^2_{\mP'}}+|Y_t-Y_s|)|Y'_t|]+ \mE^{\mP'}[|Y'_t-Y'_s|]\\
\lesssim & (1+\|\delta Y\|_{\beta;m,n})\sup_{t\in [0,T]}\|Y'_t\|_{L^m_{\mP'}}|t-s|^{\beta}+\|\delta Y'\|_{\beta';m,n}|t-s|^{\beta'}\\
\lesssim & (1+ \|(Y,Y')\|_{\bB;\beta,\beta';m,n})\|(Y,Y')\|_{\bB;\beta,\beta';m,n}|t-s|^{\beta'},
\end{align}
 which gives
 \begin{align}\label{lemma:construct:cvf:est2}
 \llb \delta \tilde \sigma'\rrb_{\beta'} \lesssim 	(1+ \|(Y,Y')\|_{\bB;\beta,\beta';m,n})\|(Y,Y')\|_{\bB;\beta,\beta';m,n}.
 \end{align}
Finally, and most importantly, we verify condition 3. To this end, we first estimate $\llb R^{\tilde \sigma^0}\rrb_{\beta+\beta'}$, in which $R^{\tilde \sigma^0}_{s,t}(x) = \sigma^0(t,x,\mu_t)-\sigma^0(s,x,\mu_s) - \tilde\sigma'_s(x)\delta B_{s,t}$. Invoking the fundamental theorem of calculus to the linear derivative $\delta \sigma^0/\delta \mu$ as well as to the derivative $\partial_x(\delta \sigma^0/\delta\mu)$ respectively, we deduce that
\begin{align}
&\sigma^0(s,x,\mu_t)-\sigma^0(s,x,\mu_s)\\
=& \int_0^1 \mE^{\mP'}\bigg[\frac{\delta \sigma^0}{\delta \mu}(s,x,\lambda \mu_t+(1-\lambda)\mu_s)(Y_t)-	\frac{\delta \sigma^0}{\delta \mu}(s,x,\lambda \mu_t+(1-\lambda)\mu_s)(Y_s)\bigg]d\lambda \\
=&\int_0^1\int_0^1\mE^{\mP'}[\partial_\mu\sigma^0(s,x,\lambda \mu_t+(1-\lambda)\mu_s)(\theta Y_t+(1-\theta)Y_s)\cdot (Y_t-Y_s)]d \theta d \lambda.
\end{align}
As a result, we have the decomposition $\sigma^0(s,x,\mu_t)-\sigma^0(s,x,\mu_s)-\tilde\sigma'_t(x)\delta B_{s,t} = I_1(s,t,x)+I_2(s,t,x)$, where
\begin{align}
I_1(s,t,x) = & \int_0^1\int_0^1 \mE^{\mP'}\Big[\Big(\partial_\mu\sigma^0\big(s,x,\lambda \mu_t+(1-\lambda)\mu_s\big)\big(\theta Y_t+(1-\theta)Y_s\big)\\
& - \partial_\mu\sigma^0(s,x,\mu_s)(Y_s) \Big)\cdot (Y_t-Y_s)\Big]d \lambda d\theta,\\
I_2(s,t,x) = & \mE^{\mP'}[\partial_\mu\sigma^0(s,x,\mu_s)(Y_s)\cdot(Y_t-Y_s-Y'_s\delta B_{s,t})]=\mE^{\mP'}[\partial_\mu\sigma^0(s,x,\mu_s)(Y_s)\cdot R^Y_{s,t}].
\end{align}
Thus, noting that
\begin{align}
\mE^{\mP'}[|\delta Y_{s,t}|^2]\leq& \mE^{\mP'}\big[ (\mE^{\mP'}_s[|\delta Y_{s,t}|^m])^{2/m}\big]\\
\leq& \big( 	\mE^{\mP'}\big[ (\mE^{\mP'}_s[|\delta Y_{s,t}|^m])^{n/m}\big]\big)^{2/n}\\
=&\|\delta Y_{s,t}\|_{m,n}^2\\
\leq &\|\delta Y\|_{\beta;m,n}^2|t-s|^{2\beta},
\end{align}
we obtain
\begin{align}
\sup_x|I_1(s,t,x)|\leq & \mE^{\mP'}[(|t-s|+\cW_2(\mu_t,\mu_s)+|Y_t-Y_s|)|Y_t-Y_s|]\\
\lesssim &\|\delta Y\|_{\beta;m,n}|t-s|^{1+\beta}+\|\delta Y\|^2_{\beta;m,n}|t-s|^{2\beta}\\
\lesssim &(1+\|(Y,Y')\|_{\bB;\beta,\beta';m,n})\|(Y,Y')\|_{\bB;\beta,\beta';m,n}|t-s|^{\beta+\beta'}.\label{lemma:construct:cvf:estI1}
\end{align}
On the other hand, it holds that
\begin{align}
\sup_x |I_2(s,t,x)| \lesssim &\|\mE^{\mP'}_{s}[R^Y_{s,t}]\|_{L^n_{\mP'}}\\
 \lesssim& \|\mE^{\mP'}_{\cdot}[R^Y]\|_{\beta+\beta';n}|t-s|^{\beta+\beta'}\\
 \lesssim& \|(Y,Y')\|_{\bB;\beta,\beta';m,n} |t-s|^{\beta+\beta'}.	\label{lemma:construct:cvf:estI2}
\end{align}
Combining \eqref{lemma:construct:cvf:estI1}-\eqref{lemma:construct:cvf:estI2} yields that
\begin{align}
\sup_x |R^{\tilde\sigma^0}_{s,t}(x)|\leq& \sup_x|\sigma^0(t,x,\mu_t)-\sigma^0(s,x,\mu_t)|	+ \sup_x |I_1(s,t,x)| + \sup_x |I_2(s,t,x)|\\
\lesssim & (1+\|(Y,Y')\|_{\bB;\beta,\beta';m,n}+\|(Y,Y')\|^2_{\bB;\beta,\beta';m,n})|t-s|^{\beta+\beta'},
\end{align}
which then gives
\begin{align}\label{lemma:construct:cvf:est3}
\llb R^{\tilde\sigma^0} \rrb_{\beta+\beta'}\lesssim 1+\|(Y,Y')\|_{\bB;\beta,\beta';m,n}+\|(Y,Y')\|^2_{\bB;\beta,\beta';m,n}.
\end{align}
Putting \eqref{lemma:construct:cvf:est1}, \eqref{lemma:construct:cvf:est2} and \eqref{lemma:construct:cvf:est3} together, we conclude \eqref{lemma:construct_cvf:est} as desired.
\end{proof}

We are now ready to define relaxed controls for the pathwise problem with rough common noise.

\begin{definition}[Relaxed control]\label{def:pathwise-relax-control}
Fix $\bB\in \hat\sC^{0,\alpha}$, $\lambda\in \cP_m(\mR^d)$, $\bmu\in \cL(\bD^{\beta,\beta'}_\bB L^{m,\infty},\lambda)$, indices $(\beta,\beta')\in \Pi$, and $m\geq 2$. A probability measure $\mP$ on $(\Omega,\F)$ is said to be an admissible relaxed control with the initial distribution $\lambda$ if:
\begin{enumerate}
	\item $(X_0)_\#\mP = \lambda$, and $X_0$ is independent of $(\Lambda,W)$ under $\mP$;
	\item Let $(\tilde\sigma^0,\tilde\sigma')\in \sD_{\bB;\beta,\beta'}$ be given from Lemma \ref{lemma:construct_cvf}. Then, with $\hat\sigma_t'(x) = \nabla \tilde \sigma^0_t(x)\tilde \sigma^0_t(x)+\tilde\sigma'_t(x)$,
    $\hat \sigma^0_t=\tilde \sigma_t^0(X_t)$, $\hat \sigma'_t = \hat \sigma'_t(X_t)$, we have $(X,\hat \sigma^0)\in \bD^{\beta,\beta'}_\bB L^{m,\infty}_{\mP}$ and $(\hat\sigma^0,\hat\sigma')\in \bD^{\beta,\beta'}_\bB L^{m,\infty}_{\mP}$.

	\item $\forall \phi \in C_0^\infty(\mR^{d+l})$, we have
	\begin{align}
            M_t(\phi): =& \phi(X_t,W_t) - \int_0^t\int_U (\widehat\mL\phi)(s,X_s,W_s,\mu_s,u) \Lambda_s(du)ds - \int_0^t \big(\Ti_s(\phi),\Ti_s'(\phi)\big) d \bB,  \\
            &- \int_0^t \Tr\big(\sigma^0(\sigma^0)^\mt(s,X_s, \mu_s)\nabla^2_x\phi(X_s,W_s)\big)d s \label{def:Mphiomega1-XW}
        \end{align}
        is a $(\mP,\mF)$-martingale, where
            \begin{align}
	&\Ti_t(\phi) :=\nabla_x \phi(X_t,W_t)\tilde \sigma_t^0(X_t)\label{eq:Ti-def} \\
	&\Ti'_t(\phi) :=\nabla^2_x\phi(X_t,W_t)(\tilde \sigma^0_t(X_t),\tilde\sigma^0_t(X_t)) +\nabla_x\phi(X_t,W_t)\hat\sigma'_t(X_t),\label{eq:Tip-def}
	\end{align}
    and the extended generator $\widehat \mL$ is given by
    \begin{align}
         \widehat\mL \phi(t,x,w,\mu,u):=& \hat b(t,x,\mu,u)^\mt \nabla \phi(x,w)+ \frac{1}{2}\Tr\bigg(\hat a(t,x,\mu)\nabla^2\phi(x,w)\bigg),\label{eq:extended-generator}
    \end{align}
    and
    \begin{align}
        \hat b(t,x,\mu,u)=\left(\begin{array}{c}
             b(t,x,\mu,u)  \\
              0 \\
        \end{array}\right), \quad \hat a(t,x,\mu) = \left(\begin{array}{c}
             \sigma(t,x,\mu)  \\
              I_{n\times n}
        \end{array}   \right)\left(\begin{array}{c}
             \sigma(t,x,\mu)  \\
              I_{n\times n}
        \end{array}   \right)^\mt. \label{eq:extend-b-sigma}
    \end{align}

\end{enumerate}
We denote by $R(\bB,\lambda,\bmu)$ the set of all admissible relaxed controls under rough-path noise $\bB$, an initial distribution $\lambda$, and an input measure flow $\bmu$. For notational convenience, we omit the dependence on $\lambda$ when no confusion arises.
\end{definition}

\begin{remark}\label{martignale-condition-special-case}
    Condition 3 in Definition \ref{def:pathwise-relax-control} implies the following two special cases, which will be used implicitly throughout this paper:
    \begin{enumerate}
        \item Taking $\phi \in C_0^\infty(\mR^l)$ to depend only on $w$, we have that
        \begin{align}
            \phi(W_t)-\frac{1}{2}\int_0^t\triangle\phi(W_s)d s, t\in [0,T],
        \end{align}
        is a $(\mP,\mF)$-martingale. By a standard approximation argument replacing $\phi$ with $\phi_i(w)=w^i$ and $\phi_{i,j}(w)=w^iw^j$, $1\leq i,j\leq l$, L\'evy's characterization implies that $W$ is a Brownian motion under $(\mP,\mF)$;
        \item Taking $\phi \in C_0^\infty(\mR^d)$ to depend only on $x$, we have
       	\begin{align}
            M^X_t(\phi): =& \phi(X_t) - \int_0^t\int_U (\mL\phi)(s,X_s,\mu_s,u) \Lambda_s(du)ds - \int_0^t \big(\Ti_s(\phi),\Ti_s'(\phi)\big) d \bB.
\label{def:Mphiomega1}
        \end{align}
        is a $(\mP,\mF)$-martingale. Here, we use the generator in $x$:
                \begin{align}
         \mL \phi(t,x,\mu,u):=& b(t,x,\mu,u)^\mt \nabla \phi(x) \\
        &+ \frac{1}{2}\Tr\bigg(\Big((\sigma \sigma^\mt)(t,x,\mu)+\big(\sigma^0(\sigma^0)^\mt\big)(t,x,\mu)\Big)\nabla^2\phi(x)\bigg).
        \end{align}
    This type of martingale condition is new to the existing literature due to the presence of rough integrals. A large part of our technical results are devoted to the treatment of $M^X(\phi)$; see Appendix \ref{app:rough-martingale}.
    \end{enumerate}
\end{remark}

\begin{remark}
    To use classical tools from stochastic analysis as well as the newly developed RSDE theory, for each {\it fixed} $\mP\in R(\bB,\bmu)$, we need to augment $\mF$ to a larger filtration satisfying the usual conditions (e.g., by right-continuous augmentation and $\mP$-completion). However, our discussion is not affected by this augmentation, because all martingales in this paper naturally admit continuous paths for {\it every} random element. We shall not explicitly point out this modification of $\mF$ throughout the paper.
\end{remark}

\begin{remark}
We point out a useful observation regarding the stochastic controlled rough path pair $(\Ti(\phi),\Ti'(\phi))$. In fact, for fixed $\phi$ (suppose for simplicity that it depends only on $x$) and $\bmu$, it can be written as the composition of a controlled vector field and another stochastic controlled rough path. Indeed, with $(\tilde\sigma^0,\tilde \sigma')$ given in Lemma \ref{lemma:construct_cvf}, let us choose $g_t(x):=\nabla \phi(x)\tilde\sigma^0_t(x)$ and $g'_t(x) = \nabla \phi(x) \tilde\sigma'_t(x)$. Then $(\Ti(\phi),\Ti'(\phi)) = (g_t,g_t')\circ(X,\hat\sigma)=(g_t,g_t')\circ (X,\tilde\sigma^0(X))$. This observation will be used in Lemma \ref{lemma:T-phi-rough-path-norm} and implicitly when translating the martingale property of $M(\phi)$ between different spaces.
\end{remark}

For a fixed rough path noise $\bB$, the objective functional of the MFG problem in \eqref{orig-obj} passes to the pathwise cost functional under the relaxed control:
\begin{align}\label{cost-pathwise}
    J_{\rm pw}(\bmu,\mP): =\mE^{\mP}\bigg[\int_0^T \int_U f(t, X_t, \mu_t,u)\Lambda_t(du)dt + g(X_T,\mu_T)\bigg], \quad \mP\in R(\bB,\lambda,\bmu).
\end{align}

We then have the following definition of Nash equilibrium for the MFG with rough common noise:

\begin{definition}\label{def:R-opt-pathwise}
If a $\mP_*\in R(\bB,\lambda,\bmu)$ is such that
\[
J_{\rm pw}(\bmu,\mP_*)=\inf_{\mP\in R(\bB,\lambda,\bmu)}J_{\rm pw}(\bmu,\mP),
\]
it is called a best response pathwise relaxed control with respect to $\bB$, $\lambda$ and $\bmu$. We denote $\mP_* \in R^{\rm opt}(\bB,\lambda,\bmu)$. Again, when there is no risk of confusion, we omit the dependence on $\lambda$.
\end{definition}

\begin{definition}\label{def:MFE-pathwise}
A pair $(\mP_*, \bmu_*)\in \cP(\Omega)\times \cL(\bD_{\bB}^{\beta,\beta'}L^{m,\infty})$ is said to be a pathwise MFE if:
\begin{enumerate}
	\item {\bf Optimality}: $\mP_*\in R^{\rm opt}(\bB,\bmu_*)$,
	\item {\bf Consistency}: $(\mu_*)_t = \cL^{\mP_*}(X_t),\forall t\in [0,T].$
\end{enumerate}
\end{definition}

As the first step to establish the existence of pathwise MFE, we argue that Definition \ref{def:pathwise-relax-control} is an appropriate definition because $R(\bB,\bmu)$ is non-empty in the subsequent section.

\subsection{Non-emptiness of \texorpdfstring{$R(\bB,\bmu)$}{R(omega1,lambda,mu)}}

We establish an enhanced existence result for any prescribed joint law $Q$ of the control distribution $(\Lambda,W)$; see Proposition \ref{lemma:nonemptyness:strong} below. As a direct consequence, $R(\bB,\lambda,\bmu)\neq \varnothing$ for any appropriate $\lambda$ and $\bmu$. This stronger result is crucial for establishing the upper hemicontinuity of the best-response mapping in Section \ref{sec:pathwiseMFE}. We also record a uniqueness result in Proposition \ref{lemma:R:uniqueness}, relying crucially on the enlarged canonical space that explicitly encodes the information of $W$.

To establish the existence result, we need to be careful about the requirement that $W$ is a Brownian motion under the overall filtration $\mF$, not just under its natural filtration. This amounts to requiring that $Q$ is a {\it causal coupling}. To state the definition precisely, we abuse notation and, when needed, treat $\Lambda$ and $W$ as coordinate mappings on $\cQ\times \cW$ (i.e., ignoring the coordinate $X$).

\begin{definition}
\begin{enumerate}
\item Define $\cP_W(\cQ\times \cW)$ to be the set of probability measures on $\cQ\times \cW$ such that $W_\#=\mP_W$, the Wiener measure on $\cW$.
\item $Q\in \cP_W(\cQ\times \cW)$ is said to be a causal coupling if for any $t\in [0,T]$, $\F^\Lambda_t$ is conditionally independent of $\F^W_T$, given $\F^W_t$, under $Q$.
\end{enumerate}
\end{definition}

A simple example of a causal coupling is open-loop control, in which case $\F^\Lambda_t\subset \F^W_t$. More generally, we also allow external randomization in $\Lambda$, as long as non-anticipation is preserved. To avoid digressions from our main topics, we collect two lemmas about causality and Brownian motion properties in Appendix \ref{app:causal-coupling}.

\begin{proposition}\label{lemma:nonemptyness:strong}
    Fix indices $(\beta,\beta')\in \Pi$, $m\geq 2$. For any causal coupling $Q\in \cP_W(\cQ\times \cW)$, $\lambda \in \cP_m(\mR^d)$, $ \bmu\in \cL(\bD^{\beta,\beta'}_\bB L^{m,\infty})$, there exists a $\mP\in R(\bB,\bmu)$ such that $(\Lambda,W)_\#\mP = Q$.
\end{proposition}

Fix an arbitrary causal coupling $Q\in \cP_W(\cQ\times \cW)$, $\lambda \in \cP_m(\mR^d)$ and $ \bmu\in \cL(\bD^{\beta,\beta'}_\bB L^{m,\infty})$. Suppose that $(\tilde\sigma^0,\tilde\sigma')\in \sD^{\beta,\beta'}_\bB C_b^\gamma$ is given by Lemma \ref{lemma:construct_cvf}. To prove Proposition \ref{lemma:nonemptyness:strong}, we use the strong solvability of RSDEs. Consider a sufficiently rich probability space $(\tilde \Omega,\tilde\mF,\tilde \F,\tilde \mP)$ supporting a $\mR^d$-valued random variable $\tilde \xi$ with distribution $\lambda$, a $\cQ$-valued random variable $\tilde \Lambda$, and a Brownian motion $\tilde W$, such that the joint law of $(\tilde \Lambda,\tilde W)$ is $Q$. Suppose in addition that $\tilde \xi$ and $(\tilde \Lambda,\tilde W)$ are independent. Such probability spaces exist thanks to Lemma~\ref{lemma:BMexistence}. Let
\begin{align}
&\bar b(t,x,\mu,\tilde \omega)=\int_Ub(t,x,\mu,u)\tilde\Lambda_t(\tilde \omega)(du)
\end{align}
Then, by Assumption \ref{ass:} (in particular, all Lipschitz coefficients and bounds are uniform in $u$), it is clear that $\bar b$ and $\sigma$ are random bounded Lipschitz functions in the sense of Definition 4.1 of \cite{RSDE}. Therefore, thanks to Theorem 4.6 of \cite{RSDE}, there exists a $L^{m,\infty}$-strong solution to the following RSDE:
\begin{align}
\begin{cases}
&d \tilde X_t = \bar b(t,\tilde X_t,\mu_t,\tilde\omega)dt +  \sigma(t,\tilde X_t,\mu_t)d\tilde W_t+(\tilde \sigma^0_t, \tilde \sigma'_t)(\tilde X_t)d\bB,\\
&\tilde X_0 =\tilde \xi.
\end{cases}
\label{RSDE:weak-solution}
\end{align}
We will prove that the joint distribution of $(\tilde X,\tilde \Lambda,\tilde W)$, denoted by $\mP$, belongs to $R(\bB,\bmu)$, yielding Proposition \ref{lemma:nonemptyness:strong}.

Before giving the proof, we need some technical preparations to handle the non-law-invariant rough path norms. Introduce the following two stochastic processes on $\tilde \Omega$: $\tilde Z_t=\tilde \sigma^0_t(\tilde X_t)$, $\tilde Z'_t = \nabla_x \tilde \sigma^0_t(\tilde X_t)\tilde \sigma^0_t(\tilde X_t)+\tilde \sigma'_t(\tilde X_t)$, and a process on $\Omega$: $\hat \sigma'_t=\nabla_x \tilde \sigma^0_t(X_t)\tilde \sigma^0_t(X_t)+\tilde \sigma'_t( X_t)$. Note also that $\hat \sigma^0_t$, defined in item 2 of Definition \ref{def:pathwise-relax-control}, can be written as $\hat \sigma^0_t = \tilde \sigma^0_t(X_t)$. The following technical result connects the stochastic controlled rough paths on $\tilde \Omega$ to those on $\Omega$. Such a connection is crucial for our results because we need to construct rough integrals on the canonical space $\Omega$.

\begin{lemma}\label{lemma:isometry}
For $\tilde U\in \{\tilde X, \tilde Z,\tilde Z'\}$ and the corresponding $\hat U\in \{X,\hat \sigma^0,\hat\sigma'\}$, we have
\begin{align}\label{eq:isometry}
\big\|\mE^{\tilde \mP}\big[|\delta \tilde U_{s,t}|^m\big|\F^{\tilde X,\tilde \Lambda,\tilde W}_s\big]^{1/m}\big\|_{\infty} = \big\|\mE^{ \mP}\big[|\delta \hat U_{s,t}|^m\big|\F_s\big]^{1/m}\big\|_{\infty},\quad \forall (s,t)\in \triangle.
\end{align}
\end{lemma}

\begin{proof}
Fix $(s,t)\in \triangle$. Combining the fact that $\hat U$ is $\mF$-adapted and the definitions of $\F^{\tilde X,\tilde \Lambda,\tilde W}_s$ and $\F_s(=\F^{X,\Lambda,W}_s)$, we have the existence of measurable functions $G$ and $\tilde G$ from $\cX\times \cQ\times \cW$ to $\mR$ such that
\begin{align}
	\mE^{\tilde \mP}\big[|\delta \tilde U_{s,t}|^m\big|\F^{\tilde X,\tilde \Lambda,\tilde W}_s\big]=\tilde G(\tilde X_{\cdot\wedge s},\tilde \Lambda|_{[0,s]\times U},\tilde W_{\cdot \wedge s}), \quad \mE^{ \mP}\big[|\delta \hat U_{s,t}|^m\big|\F_s\big]=G(X_{\cdot\wedge s},\Lambda|_{[0,s]\times U},W_{\cdot \wedge s}).
\end{align}
Note that, for any measurable function $\varphi$ on $\cX\times \cQ$,
\begin{align}
	&\mE^{\mP}[G(X_{\cdot\wedge s},,\Lambda|_{[0,s]\times U},W_{\cdot \wedge s})\varphi(X_{\cdot \wedge s},\Lambda|_{[0,s]\times U},W_{\cdot \wedge s})] \\
    \stackrel{\text{(i)}}{=} & \mE^{\mP}[|\delta \hat U_{s,t}|^m \varphi(X_{\cdot \wedge s},\Lambda|_{[0,s]\times U},W_{\cdot \wedge s})] \stackrel{\text{(ii)}}{=} \mE^{\tilde \mP}[|\delta \tilde U_{s,t}|^m \varphi(\tilde X_{\cdot \wedge s},\tilde \Lambda|_{[0,s]\times U},\tilde W_{\cdot \wedge s})]\\
	\stackrel{\text{(iii)}}{=}& \mE^{\tilde \mP}[\tilde G(\tilde X_{\cdot \wedge s},\tilde \Lambda|_{[0,s]\times U},\tilde W_{\cdot \wedge s})\varphi(\tilde X_{\cdot \wedge s},\tilde \Lambda|_{[0,s]\times U},\tilde W_{\cdot \wedge s})]\\
	\stackrel{\text{(iv)}}{=}& \mE^{\mP}[\tilde G( X_{\cdot \wedge s},\Lambda|_{[0,s]\times U},W_{\cdot \wedge s})\varphi(X_{\cdot \wedge s},\Lambda|_{[0,s]\times U},W_{\cdot \wedge s})]. \label{proof:lemma:isometry:eq-rep}
\end{align}
Here, in equalities (ii) and (iv), we use the change-of-variable formula to shift integrations between the two spaces $\tilde \Omega$ and $\Omega$, and in (i) and (iii), we use the orthogonality of conditional expectation. From \eqref{proof:lemma:isometry:eq-rep} we conclude that $G(X_{\cdot \wedge s},\Lambda|_{[0,s]\times U},W_{\cdot \wedge s})=\tilde G(X_{\cdot \wedge s},\Lambda|_{[0,s]\times U},W_{\cdot \wedge s})$, $ \mP$-a.s.. As a consequence, for any $a>0$,
\begin{align}
	\tilde \mP\big(\tilde G(\tilde X_{\cdot \wedge s},\tilde \Lambda|_{[0,s]\times U},\tilde W_{\cdot \wedge s})>a  \big)= \mP\big( \tilde G(X_{\cdot \wedge s},\Lambda|_{[0,s]\times U},W_{\cdot \wedge s})>a\big) =  \mP\big( G(X_{\cdot \wedge s},\Lambda|_{[0,s]\times U},W_{\cdot \wedge s})>a\big).
\end{align}
Employing the $\infty$-norm, we obtain \eqref{eq:isometry}.
\end{proof}

\begin{proof}[Proof of Proposition \ref{lemma:nonemptyness:strong}]
	As already mentioned, we shall prove that the distribution of the solution to \eqref{RSDE:weak-solution} belongs to $R(\bB,\bmu)$. Condition 1 of Definition \ref{def:pathwise-relax-control} is clear. Let us verify conditions 2 and 3 in the following two main steps.

{\it Verifying condition 2:} We note that, for $\tilde U\in \{\tilde X, \tilde Z,\tilde Z'\}$, by tower property of the conditional expectation,
	\begin{align}
	\mE^{\tilde \mP}[|\delta \tilde U_{s,t}|^m|\F^{\tilde X,\tilde \Lambda,\tilde W}_s]^{1/m}=& 	\mE^{\tilde \mP}\big[\mE^{\tilde \mP}[|\delta \tilde U_{s,t}|^m|\tilde \F_s]\big|\F^{\tilde X,\tilde \Lambda, \tilde W}_s\big]^{1/m}\\
	\leq & \|\mE^{\tilde \mP}[|\delta \tilde U_{s,t}|^m|\tilde \F_s]\|^{1/m}_\infty.\\
    =&\|\mE^{\tilde \mP}[|\delta \tilde U_{s,t}|^m|\tilde \F_s]^{1/m}\|_\infty
	\end{align}
Thus, by Lemma \ref{lemma:isometry},
\begin{align}\label{eq:change-the-filtration}
	\|\mE^{ \mP}[|\delta \hat U_{s,t}|^m|\F_s]^{1/m}\|_{\infty}\leq \|\mE^{\tilde \mP}[|\delta \tilde U_{s,t}|^m|\tilde \F_s]^{1/m}\|_{\infty}.
\end{align}
On the other hand, we have $\|\mE^{\tilde\mP}_{\cdot}[R^{\tilde Z}]\|_{\infty}=\|\mE^{\check\mP}_{\cdot}[R^{\hat\sigma}]\|_{\infty}$ and $\|\mE^{\tilde\mP}_{\cdot}[R^{\tilde Z'}]\|_{\infty}=\|\mE^{\check\mP}_{\cdot}[R^{\hat\sigma'}]\|_{\infty}$. By the definition in \eqref{def:norm:scrp}, we conclude
\begin{align}
\|(\hat \sigma^0,\hat \sigma')\|_{\bB;\beta,\beta';m,\infty}\leq \|(\tilde Z,\tilde Z')\|_{\bB;\beta,\beta';m,\infty}.
\end{align}
Combining an a priori estimate for RSDEs (c.f., Proposition 4.5 of \cite{RSDE}) and \eqref{lemma:construct_cvf:est}, we get
\begin{align}
    &\|(X,\hat\sigma^0)\|_{\bB;\beta,\beta';m,\infty}\lesssim M_{\bmu}(1+M_{\bmu}\|\bB\|_{\alpha})^{1/\beta'}, \label{X-rough-path-norm}\\
	&\|(\hat \sigma^0,\hat \sigma')\|_{\bB;\beta,\beta';m,\infty}\lesssim M_{\bmu}^2(1+M_{\bmu}\|\bB\|_{\alpha})^{(\gamma-1)/\beta'},\label{sigma0-rough-path-norm}
\end{align}
with
\begin{align}
M_{\bmu} = 1+\|(Y,Y')\|_{\bB;\beta,\beta';m,n}+\|(Y,Y')\|^2_{\bB;\beta,\beta';m,n}.
\end{align}

{\it Verifying condition 3:} To this end, we fix a test function $\phi\in C_0^\infty(\mR^{d+l})$. By \eqref{X-rough-path-norm}, \eqref{sigma0-rough-path-norm} and the assumptions on $\sigma^0$, we have $(\Ti(\phi),\Ti'(\phi))\in \bD^{\beta,\beta'}_{\bB}L^{m,\infty}_{\mP}$. Consequently, the rough integral on the right-hand side of \eqref{def:Mphiomega1} is well defined. Moreover, for any $t\in [0,T]$, it is the limit in $\mP$-probability, as $|\cP|\to 0$, of the following Riemann sum (c.f. Theorems 2.8 and 3.5 of \cite{RSDE}):
\begin{align}
A^{\cP}_t:=\sum_{[u,v]\in \cP,0\leq u< t} (\Ti_u(\phi)\delta B_{u,v}+\Ti'_u(\phi)\mB_{u,v}),
\end{align}
	where $\cP$ is a partition of $[0,T]$ and $|\cP|$ is the mesh size. On the other hand, recall the notation $(\tilde \Omega,\tilde \F,\tilde \mF, \tilde \mP,\tilde W,\tilde X)$ for the weak solution to \eqref{RSDE:weak-solution}, as well as the definition of the processes $\tilde Z$ and $\tilde Z'$ on $\tilde \Omega$. Let us consider
\begin{align}
  \tilde M_t(\phi): =& \phi(\tilde X_t,\tilde W_t) - \int_0^t \int_U(\widehat\mL\phi)(s,\tilde X_s,\tilde W_s,\mu_s,u)\tilde\Lambda_s(du) ds - \int_0^t (\tilde \Ti_\phi,\tilde \Ti_\phi') d \bB,\label{def:Mphitilde}
\end{align}
where
\begin{align}
    &\tilde \Ti_t(\phi) := \nabla \phi(\tilde X_t,\tilde W_t) \tilde Z_t,\\
	&\Ti'_t(\phi) :=  \nabla^2\phi(\tilde X_t,\tilde W_t)(\tilde Z_t,\tilde Z_t) + \nabla\phi(\tilde X_t,\tilde W_t)\tilde Z_t'.
\end{align}
By the rough (stochastic) It\^o formula (Theorem 4.13 of \cite{RSDE}), $\tilde M_\phi$ is a $(\tilde \mP,\tilde \mF)$-martingale because it can be expressed as an It\^o integral against $\tilde W$\footnote{Here, we actually apply the rough It\^o formula to $(\tilde X, \tilde W)$. It of course satisfies an RSDE with coefficients specified by \eqref{eq:extend-b-sigma}.}.

Now, introduce the Riemann sum on $\tilde \Omega$:
    \begin{align}
\tilde A^{\cP}_t=\sum_{[u,v]\in \cP,0\leq u< t} (\tilde\Ti_u(\phi)\delta B_{u,v}+\tilde\Ti'_u(\phi)\mB_{u,v}).
    \end{align}
By definition of $\mP$, we have
\begin{align}
    \cL^{\tilde \mP}(\tilde A^\cP_t,\tilde X,\tilde \Lambda,\tilde W) = \cL^{\mP}(A^{\cP}_t,X,\Lambda,W), \quad \forall t\in [0,T].
\end{align}
Using convergence in probability (and hence convergence in law), we obtain
\begin{align}
\cL^{\tilde \mP}\bigg(\int_0^t\big(\tilde \Ti(\phi),\tilde \Ti'(\phi)\big)d\bB,\tilde X,\tilde \Lambda,\tilde W\bigg)=\cL^{\mP}\bigg(\int_0^t\big( \Ti(\phi), \Ti'(\phi)\big)d\bB, X,\Lambda,W\bigg).
\end{align}
For the other parts of $\tilde M_\phi$ and $M_\phi$, we also have a similar distributional identity because they are classical Lebesgue integrals. In sum, we have shown that $\cL^{\tilde \mP}(\tilde M_t(\phi),\tilde X,\tilde \Lambda,\tilde W)=\cL^{\mP}(M_t(\phi),X,\Lambda,W)$ for each $t\in [0,T]$. Now, for any $0\leq s<t\leq T$, choosing an arbitrary bounded measurable function $\varphi$ on $\Omega$ that is $\F_s$-measurable, we conclude from the martingale property of $\tilde M_\phi$ that
\begin{align}
    \mE^{\mP}[M_t(\phi)\varphi(X_{\cdot \wedge s}, \Lambda|_{[0,s]\times U},W_{\cdot \wedge s})]=&\mE^{\tilde \mP}[\tilde M_t(\phi)\varphi(\tilde X_{\cdot \wedge s},\tilde \Lambda|_{[0,s]\times U},\tilde W_{\cdot \wedge s})] \\
    =&\mE^{\tilde \mP}[\tilde M_s(\phi)\varphi(\tilde X_{\cdot \wedge s},\tilde \Lambda|_{[0,s]\times U},\tilde W_{\cdot \wedge s})] \\
    =&\mE^{\mP}[ M_s(\phi)\varphi( X_{\cdot\wedge s}, \Lambda|_{[0,s]\times U},W_{\cdot \wedge s})] ,
\end{align}
verifying the martingale property of $M_\phi$, which completes the proof.
\end{proof}

\section{Existence of Pathwise MFE}
\label{sec:pathwiseMFE}

\subsection{Step-1: compactness of \texorpdfstring{$R(\bB,\bmu)$}{R(omega1,lambda,mu)}}
\label{subsec:step1}

This subsection proves the compactness of \texorpdfstring{$R(\bB,\bmu)$}, which plays a key role in the proof of existence of Nash equilibrium in the pathwise problem.

\begin{proposition}\label{prop:compact}
Fix indices $(\beta,\beta')\in \Pi$, $m\geq 4$. For any $\lambda\in \cP_2(\mR^d)$ and $\bmu\in \cL(\bD^{\beta,\beta'}_\bB L^{m,\infty})$, $R(\bB,\bmu)\subset \cP(\Omega)$ is compact in the topology of weak convergence.
\end{proposition}

For simplicity, let us introduce the following notations: for $\phi\in C_0^\infty(\mR^d)$ (a test function only in $x$), $t\in[0,T]$,
\begin{align}
    &I_t(\phi) = \int_0^t \int_U \mL\phi(s,X_s,\mu_s,u)\Lambda_s(du)ds,\\
    &\cA_t(\phi) = \int_0^t \big(\Ti_s(\phi),\Ti'_s(\phi)\big)d \bB.
\end{align}

To prove Proposition \ref{prop:compact}, we need several auxiliary preparations for the rough martingale problem of $M^X$ (see \eqref{def:Mphiomega1}), which are collected in Appendix \ref{app:rough-martingale}. In addition, the following lemma is needed in the subsequent analysis.

\begin{lemma}\label{lemma:T-phi-rough-path-norm}
    For $\phi\in C_0^\infty(\mR^d)$, define $\|\nabla \phi\|_{C^2}:= \|\nabla \phi\|_\infty+\|\nabla^2\phi\|_\infty+\|\nabla^3\phi\|_\infty$. Then,
    \begin{align}\label{T-phi-rough-path-norm}
        \|(\Ti(\phi),\Ti'(\phi))\|_{\bB;\beta',\beta;m,\infty}\lesssim \|\nabla \phi\|_{C^2}\llb (\tilde\sigma^0,\tilde \sigma')\rrb_{\bB,\beta,\beta'}(1+\|(X,\hat \sigma^0)\|_{\bB,\beta,\beta';m,\infty}^{\gamma}).
    \end{align}
\end{lemma}

\begin{proof}
    Fix $\phi$ throughout this proof. Let us consider $g_t(x):=\nabla \phi(x) \tilde \sigma^0_t(x) \in \mR^{k}$ and it follows that
    \begin{align}
        \nabla g_t(x) = \nabla^2\phi(x)(\tilde \sigma^0_t(x),\cdot)+\nabla \phi(x)\nabla \tilde \sigma^0_t(x) \in \mR^{ k\times d}.
    \end{align}
    Note that
    \begin{align}
        \Ti'_t(\phi)=&\nabla^2\phi(X_t)(\tilde \sigma^0_t(X_t),\tilde \sigma^0_t(X_t)) + \nabla\phi(X_t)\hat\sigma'_t\\
        =&\nabla g_t(X_t)\tilde\sigma^0_t(X_t) + \nabla\phi(X_t)\big(\hat \sigma'_t - \nabla\tilde\sigma^0_t(X_t)\tilde \sigma^0_t(X_t)\big)\\
        =&\nabla g_t(X_t)\tilde\sigma^0_t(X_t) + \nabla\phi_t(X_t)\tilde\sigma'_t(X_t).
    \end{align}
    Hence, choosing $g'_t(x) = \nabla \phi(x) \tilde\sigma'_t(x)$, we have $(\Ti(\phi),\Ti'(\phi)) = (g,g')\circ(X,\hat \sigma^0)$\footnote{We emphasize that the composition here is the sense of controlled vector field, i.e., $(f,f')\circ(X,X'):=(f(X),\nabla f(X)X'+f'(X))$.}.
    In view of the fact $\llb (g,g')\rrb_{\bB,\beta,\beta'}\lesssim \|\nabla \phi\|_{C^2}\llb (\tilde\sigma^0,\tilde \sigma')\rrb_{\bB,\beta,\beta'}$, \eqref{T-phi-rough-path-norm} is a consequence of Lemma 3.11 of \cite{RSDE}.
\end{proof}

As an important step toward tightness, we establish the equivalence between the rough martingale problem (i.e., the condition $\mP\in R(\bB,\bmu)$) and the RSDE in the next result, which is important in its own right.

\begin{proposition}\label{prop:equiv}
    For $\mP\in R(\bB,\bmu)$, it holds $\mP$-a.s. that
    \begin{align}\label{RSDE-for-PinR}
        X_t=&X_0+\int_0^t\int_U b(s,X_s,\mu_s,u)\Lambda_s(du)ds+\int_0^t\sigma(s,X_s,\mu_s) dW_s \\
        &+\int_0^t(\tilde\sigma^0,\tilde \sigma')\circ (X,\tilde\sigma^0(X)) d\bB, \quad \forall t\in [0,T].
    \end{align}
\end{proposition}

\begin{proof}
     Choose a vector-valued test function $\vec \phi\in C_0^\infty(\mR^d;\mR^d)$. Define $M^X(\vec\phi)$, $I(\vec\phi)$ and $\cA(\vec\phi)$ componentwise, so that
     \begin{align}\label{eq:RSDE:phi}
         \vec\phi(X_t)=\vec\phi(X_0)+I_t(\vec\phi)+\cA_t(\vec\phi)+M^X_t(\vec\phi).
     \end{align}
     Utilizing the cut-off function, we obtain the existence of a sequence of $\vec\phi_n\in C_0^\infty(\mR^d;\mR^d)$ satisfying:
    \begin{enumerate}
        \item $\vec\phi_n(x)\to x$, $\nabla\vec\phi_n(x)\to I_{d\times d}$, $\nabla^2\vec\phi_n(x)\to 0$, pointwise;
        \item $\{(\nabla\vec\phi_n,\nabla^2\vec\phi_n)\}_{n\geq 1}$ is uniformly bounded on $\mR^d$ (in suitable operator norms for matrices and tensors).
    \end{enumerate}
   By virtue of the boundedness of coefficients, we get the convergence of $I(\vec\phi_n)$:
    \begin{align}
        \sup_{t\in [0,T]}\bigg|\int_0^t \int_U b(s,X_s,\mu_s,u)\Lambda_s(du)ds - I_t(\vec\phi_n)\bigg|\to 0, \quad {\rm in\ \ } L^2(\Omega;\mP).
    \end{align}
    As a direct result of Lemma \ref{T-phi-rough-path-norm}, $\|(\Ti(\phi),\Ti'(\phi))\|_{\bB;\beta',\beta;m,\infty}$ is uniformly bounded, and then Lemma 4.20 of \cite{RSDE} yields that
    \begin{align}
        \sup_{t\in [0,T]}\bigg|\int_0^t(\tilde\sigma^0,\tilde \sigma')\circ (X,\tilde\sigma^0(X)) d\bB - \cA_t(\vec\phi_n)\bigg|\to 0, \quad {\rm in\ \ } L^2(\Omega;\mP).
    \end{align}
    To handle the martingale term in \eqref{eq:RSDE:phi}, we consider the space $\M^{2,c}_T$ of square-integrable, continuous martingales on $(\Omega,\F_T,\mF,\mP)$ with time index $t\in [0,T]$, equipped with the metric:
    \begin{align}
        \rho_{\M}(M^1,M^2):= (\mE^\mP[|M^1_T-M^2_T|^2])^{1/2}=(\mE^\mP[\langle M^1-M^2\rangle_T])^{1/2}.
    \end{align}
    By Proposition 1.5.23 of \cite{karatzas1991brownian}, $(\M^{2,c}_T,\rho_\M)$ is complete. Write $\vec\phi=(\phi^1,\phi^2,\cdots, \phi^d)$. Then, for each $j=1,2,\cdots ,d$, $n\geq 1$, it is clear that $M^X(\phi^j_n)\in \M^{2,c}_T$. From Lemma \ref{lemma:Mphi:QV}, it follows that
    \begin{align}
        \rho_M\big(M^X(\phi^j_n),M^X(\phi^j_m)\big)^2=\mE^\mP\bigg[\int_0^T \big|\big(\nabla \phi^j_n(X_s)-\nabla \phi^j_m(X_s)\big)\sigma(s,X_s,\mu_s)|^2 ds\bigg] \to 0,\quad m,n\to \infty.
    \end{align}
    Hence there exists an $M^j\in \M^{2,c}_T$ such that $\rho_\M(M^j,M^X(\phi^j_n))\to 0$. Furthermore, the bi-linearity of $\langle \cdot,\cdot\rangle$ implies that
    \begin{align}\label{Mij}
        \langle M^i,M^j\rangle = \langle M^i-M^X(\phi^i_n),M^j\rangle + \langle M^X(\phi^i_n),M^j-M^X(\phi^j_n)\rangle + \langle M^X(\phi^i_n),M^X(\phi^j_n)\rangle.
    \end{align}
    By virtue of Remark \ref{rmk:cross-variation} and the Cauchy-Schwarz inequality, we get that
    \begin{align}\label{M-bracket}
        \langle M\rangle_t=\int_0^t \sigma\sigma^\mt(s,X_s,\mu_s) ds, \quad \forall t\in [0,T],
    \end{align}
    holds $\mP$-almost surely, in which $M=(M^1,M^2,\cdots, M^d)$. Here, to simplify the notation we define $\langle M\rangle$ for a $\mR^d$-valued martingale $M$ to be the $d\times d$ matrix, whose $(i,j)$-element is $\langle M^i,M^j\rangle$. Letting $n\to \infty$ in \eqref{eq:RSDE:phi} and taking the limit in $L^2(\Omega;\mP)$, one obtains
    \begin{align}\label{eq:RSDE:M}
        X_t = X_0 +\int_0^t \int_U b(s,X_s,\mu_s,u)\Lambda_s(du)ds+\int_0^t(\tilde\sigma^0,\tilde \sigma')\circ (X,\tilde\sigma^0(X)) d\bB + M_t
    \end{align}

    On the other hand, for $j=1,2,\cdots, l$, using the cut-off functions again yields a sequence $\psi^n_j\in C_0^\infty(\mR^l)$ with uniformly bounded gradients and Hessians, and $\psi^j_n(w)\to w^j$, $n\to \infty$. Using Lemma \ref{lemma:cross-M-W} and letting $n\to \infty$ in $\langle M^X(\phi^i_n),M^W(\psi^j_n)\rangle$ for $i=1,2,\cdots, d$, $j=1,2,\cdots, l$, we get from a similar relation as in \eqref{Mij} that
    \begin{align}\label{MW-bracket}
        \langle M^i,W^j\rangle_t=\int_0^t \sigma_{ij}(s,X_s,\mu_s) ds.
    \end{align}
    Combining
    \eqref{M-bracket} and \eqref{MW-bracket} gives
    \begin{align}
        M_t=\int_0^t \sigma(s,X_s,\mu_s) dW_s.
    \end{align}
    Plugging this back into \eqref{eq:RSDE:M} completes the proof.
\end{proof}

\begin{remark}\label{remark:uniqueness}
    The equivalence result in Proposition \ref{prop:equiv} is used several times in this paper. First, it directly gives tightness in Lemma \ref{lemma:tight:1} (see also Corollary \ref{lemma:tight} below). Second, it is crucial to the proof of Proposition \ref{lemma:R:uniqueness} and Lemma \ref{lemma:Phi:continuity}, which is a key step toward the existence of pathwise MFE. Moreover, Proposition \ref{prop:equiv} justifies our martingale formulation in Definition \ref{def:pathwise-relax-control}. Indeed, for any $\mP\in R(\bB,\bmu)$, we can always find a state process in the form of an RSDE, with the same control-noise coupling on $(\Lambda,W)$, that produces a state process with distribution $(X)_\#\mP$. After randomization (see Section \ref{sec:randomization}), this leads to the conventional MFG problem.
\end{remark}

As a direct result, the following a priori estimate holds for all $\mP\in R(\bB,\bmu)$.

\begin{corollary}\label{lemma:tight}
    There exists a $\gamma>0$ such that, for each $\mP\in R(\bB,\bmu)$ and $s<t$,
    \begin{align}\label{eq:aprioriest}
       \|(X,\hat\sigma^0)\|_{\bB;\alpha,\beta;m,\infty;[s,t];\mP}\lesssim (1\vee\llbracket(\tilde \sigma^0,\tilde \sigma')\rrbracket_{\bB;\beta,\beta';[s,t]})^{\gamma}.
    \end{align}
\end{corollary}

\begin{proof}
    This is a corollary of Proposition 4.5 in \cite{RSDE} and Proposition \ref{prop:equiv} above. It is worth noting that their estimates naturally hold for any subinterval $[s,t]\subset [0,T]$, with hidden constants non-decreasing in $t-s$.
\end{proof}

We are now ready to present the proof of Proposition \ref{prop:compact}. For clarity, we split it into two sub-results: the tightness and closedness.

\begin{lemma}\label{lemma:tight:1}
    $R(\bB,\bmu)\subset \cP(\Omega)$ is tight.
\end{lemma}

\begin{proof}
 As a direct consequence of \eqref{eq:aprioriest}, we get a bound on $\|\delta X\|_{\beta;m,\infty}$ that is independent of $\mP$. In particular, for $s<t$,
\begin{align}\label{eq:aldous}
    \mE^{\mP}[|X_t-X_s|^m] \leq \|\delta X\|_{\beta;m,\infty}^m (t-s)^{m\beta},
\end{align}
and $m\beta\geq 4\beta>1$. By Aldous' criterion and the Kolmogorov continuity theorem (e.g., Theorem 3.1 of \cite{rough_path_book}), it follows that $\{X_\#\mP:\mP\in R(\bB,\bmu)\}$ is tight. Then, since $U$ is compact by assumption, so is $\cQ\subset \cP_T([0,T]\times U)$, and consequently $\{\Lambda_\#\mP:\mP\in R(\bB,\bmu)\}$ is tight. Since the $W$-marginal of $\mP$ is fixed, it is standard to combine tightness of the marginals to obtain tightness of $R(\bB,\bmu)$ (e.g., Proposition 2.4 of \cite{EthierKurtz1986Markov}).
\end{proof}

\begin{remark}
    Although $\|(X,\hat\sigma^0)\|_{\bB;\beta,\beta';m,\infty}<\infty$ (for {\it each fixed} $\mP\in R(\bB,\bmu)$) is part of the definition of $R(\bB,\bmu)$, Corollary \ref{lemma:tight} provides a uniform-in-$\mP$ bound, which is central to our analysis. By Lemma \ref{lemma:construct_cvf} and Lemma 3.11 of \cite{RSDE}, it is clear that $\|(\hat\sigma^0,\hat\sigma')\|_{\bB;\beta,\beta';m,\infty}$ also has a uniform-in-$\mP$ bound.
\end{remark}

\begin{lemma}\label{lemma:pathwise:closedness}
    $R(\bB,\bmu)\subset \cP(\Omega)$ is closed under the topology of weak convergence.
\end{lemma}

\begin{proof}
Suppose that $\{\mP^n\}_{n\geq 1}\subset R(\bB,\bmu)$ and $\mP^n\to \mP\in \cP(\Omega)$ weakly. By Skorohod representation theorem, there exist a probability space $(\Omega',\F',\mP')$ and $\Omega$-valued random variables $(X^n,\Lambda^n,W^n)$ and $(X,\Lambda,W)$ such that $\cL^{\mP'}(X^n,\Lambda^n,W^n) = \cL^{\mP^n}(X,\Lambda,W)$, $\cL^{\mP'}(X',\Lambda',W') = \cL^{\mP}(X,\Lambda,W)$,\; and $(X^{n},\Lambda^n,W^n)\to (X',\Lambda',W')$, $\mP'$-a.s.. Note that the convergence of elements on $\Omega$ is interpreted as being under $d_\infty$ in $\cX$ and $\cW$, and weak convergence in $\cQ$. In particular, for any $t\in [0,T]$, $X^n_{\cdot\wedge t}\to X'_{\cdot \wedge t}$, $W^n_{\cdot\wedge t}\to W'_{\cdot \wedge t}$, $\Lambda^n|_{[0,t]\times U}\to \Lambda'|_{[0,t]\times U}$ $\mP'$-a.s., the first two under the uniform distance last one under stable convergence on $\cQ$. We now verify that $\mP$ satisfies Definition \ref{def:pathwise-relax-control}.

First, condition 1 is clearly satisfied. We next verify condition 2. The main difficulty here is that the rough path norms are not law-invariant. It is more helpful to analyze the explicit dependence on probability measures and filtrations. In particular, we use $\|\cdot\|_{\bB;\beta,\beta';m,\infty;\mP,\mF}$ to denote the corresponding stochastic rough path norms for some probability measure $\tilde \mP$ and some filtration $\tilde \mF$. The probability space on which these norms are defined will be clear from the notation. On $\Omega'$, we denote by $\mF'$ the filtration generated by $(X',\Lambda',W')$, and $\mF^n$ the filtration generated by $(X^n,\Lambda^n,W^n)$, both completed by $\mP'$. Using the invariance property established in Lemma \ref{lemma:isometry}, along with the uniform estimate in Corollary~\ref{lemma:tight}, we deduce that, $\mP'$-a.s.,
\begin{align}
   \mE^{\mP'}[|X^n_t-X^n_s|^m|\F^n_s]\leq& \|(X,\hat\sigma^0)\|^m_{\bB;\beta,\beta';m,\infty;\mP^n,\mF^n}(t-s)^{m\beta}\\
    \leq & C(t-s)^{m \beta},
\end{align}
with a constant $C$ that is independent of $n$. On the other hand, it is clear that the family of random variables $\{X^n_t-X^n_s\}_{n\geq 1}$ is uniformly integrable. Thus, for any bounded positive continuous function $\varphi$ defined on $\Omega$ (which only depends on the path up to $s$), we have
\begin{align}
  \mE^{\mP'}[|X'_t-X'_s|^m\varphi(X'_{\cdot\wedge s},\Lambda'|_{[0,s]\times U},W'_{\cdot\wedge s})]\leq &\liminf_{n\to \infty}\mE^{\mP'}[|X^n_t-X^n_s|^m \varphi(X^n_{\cdot\wedge s},\Lambda^n|_{[0,s]\times U},W^n_{\cdot\wedge s})]\\
  =&\liminf_{n\to \infty}\mE^{\mP'}\big[\mE^{\mP'}[|X^n_t-X^n_s|^m |\F^n_s]\varphi(X^n_{\cdot\wedge s},\Lambda^n|_{[0,s]\times U})\big]\\
  \leq&C(t-s)^{m\beta}\liminf_{n\to \infty}\mE^{\mP'}[\varphi(X^n_{\cdot\wedge s},\Lambda^n|_{[0,s]\times U},,W^n_{\cdot\wedge s})]\\
  =&C(t-s)^{m\beta}\mE^{\mP'}[\varphi(X'_{\cdot\wedge s},\Lambda'|_{[0,s]\times U},W'_{\cdot\wedge s})].
\end{align}
Consequently, $\mE^{\mP'}[|X'_t-X'_s|^m|\F'_s]\leq C(t-s)^{m\beta}$, $\mP'$-a.s., thus $\|\delta X'\|_{\beta;m,\infty}\leq C^{1/m}$. Moreover, $\tilde\sigma^0$ and $\hat\sigma'$ are both continuous, the above argument can be repeated to derive similar bounds. We therefore conclude that, $\|(X',\tilde\sigma^0(X'))\|_{\bB;\beta,\beta';m,\infty;\mP',\mF'}$ and $\|(\tilde\sigma^0(X'),\hat\sigma'(X'))\|_{\bB;\beta,\beta';m,\infty;\mP',\mF'}$ also satisfy the estimate derived in Lemma \ref{lemma:tight}. Utilizing the invariance property in Lemma \ref{lemma:isometry} again, condition 2 can be verified.

To show that condition 3 holds, we use an argument similar to the third step in the proof of Proposition \ref{lemma:nonemptyness:strong} to transfer the martingale property from $\Omega$ to $\Omega'$. More precisely, by a similar Riemann-sum approximation, we have
\begin{align}
\cL^{ \mP'}( M^n_t(\phi),X^n,\Lambda^n, W^n)=\cL^{\mP^n}(M_t(\phi),X,\Lambda,W)
\end{align}
for each $t\in [0,T]$. Here $ M^n(\phi)$ is defined similarly as in \eqref{def:Mphiomega1-XW}, with $(X, \Lambda, W)$ replaced by $( X^n, \Lambda^n, W^n)$ (though on a different probability space $(\Omega',\mF',\F',\mP')$). As $\mP^n\in R(\bB,\bmu^n)$, we know $ M^n(\phi)$ is a $(\mP', \mF')$-martingale on $\Omega'$. In particular, for any bounded continuous function $\varphi$ on $\Omega$, only depending on $(X_{\cdot \wedge s},\Lambda|_{[0,s]\times U},W_{\cdot \wedge s})$, we have
\begin{align}\label{martingale:Mn}
    \mE^{\mP'}[M^n_t(\phi)\varphi(X^n_{\cdot \wedge s}, \Lambda^n|_{[0,s]\times U},W^n_{\cdot \wedge s})]=\mE^{\mP'}[ M^n_s(\phi)\varphi( X^n_{\cdot\wedge s}, \Lambda^n|_{[0,s]\times U},W^n_{\cdot \wedge s})] ,
\end{align}
Combining the estimate of $X^n$ and $X'$ above, and using Lemma 4.20 of \cite{RSDE} to handle the rough integral term, we have the following convergence:
\begin{align}
    \lim_{n\to \infty}\sup_{t\in [0,T]}|M^n_t(\phi)-M'_t(\phi)|\to 0, \quad {\rm in\ \ } L^m(\Omega',\F',\mP'),
\end{align}
with $M'(\phi)$ defined similarly as $M^n(\phi)$ by replacing $(X^n,\Lambda^n,W^n)$ with $(X',\Lambda',W')$. Sending $n\to \infty$ in \eqref{martingale:Mn} gives
\begin{align}
    \mE^{\mP'}[M'_t(\phi)\varphi(X'_{\cdot \wedge s}, \Lambda'|_{[0,s]\times U},W'_{\cdot \wedge s})]=\mE^{\mP'}[ M'_s(\phi)\varphi( X'_{\cdot\wedge s}, \Lambda'|_{[0,s]\times U},W'_{\cdot \wedge s})].
\end{align}
Because $\cL^{\mP'}(M'_t(\phi),X',\Lambda',W')=\cL^\mP(M_t(\phi),X,\Lambda,W)$ for each $t\in [0,T]$, we arrive at the martingale condition for $M(\phi)$ defined in \eqref{def:Mphiomega1-XW}.
\end{proof}

To establish the existence of pathwise MFE, we need to apply the Kakutani's fixed point theorem to the set-valued mapping
\begin{align}
    \Phi: \bmu \mapsto \{\bnu=X_\#\mP:\mP\in R^{\rm opt}\big(\bB,\bmu\big)\}.
\end{align}
To this end, we split our proof into three main steps elaborated in the following subsections.

\subsection{Step-2: identify the domain of $\Phi$}
\label{subsec:step-2}
Consider the following domain of $\bmu$:
\begin{align}
     \cP_{M,\epsilon} = \{\bmu\in \cL(\bD^{\beta,\beta'}_\bB L^{m,\infty}): \|(Y,Y')\|_{\bB;\beta,\beta';m,\infty;[s,t];\mP}\leq M,\forall s,t{\rm\ with\ }t-s<\epsilon,\\
        {\rm\ for\ some\ representative\ }(Y,Y') \}.\label{eq:def:domain}
\end{align}
We will choose $M,\epsilon$ later on to ensure that $\Phi$ is invariant on $\cP_{M,\epsilon}$. Here, for the sake of convexity, we use a slight variation of rough path norms in this section:
\begin{align}
    \|(Y,Y')\|_{\bB;\beta,\beta';m,\infty;[s,t];\mP}:=\bigg(\frac{1}{3}(\|\delta Y\|^m_{\beta;m.\infty;[s,t];\mP}+\|Y'\|^m_{\beta';m,\infty;[s,t];\mP}+\|\mE_\cdot [R^Y]\|^m_{\beta+\beta';\infty;[s,t];\mP})\bigg)^{1/m}.
\end{align}
This revision preserves all estimates in the previous subsections, up to a constant depending only on $m$.

\begin{remark}
    We can extend the local bound in the definition of $\cP_{M,\epsilon}$ to a global bound. More precisely, using the argument in Step 2 of the proof of Proposition 4.5 in \cite{RSDE}, we easily obtain a constant $C$, depending only on $M,\epsilon, T,m$, such that
    \begin{align}
         \|(Y,Y')\|_{\bB;\beta,\beta';m,\infty}\leq C,
    \end{align}
    for any $\bmu\in \cP_{M,\epsilon}$ and any representation $(Y,Y')$ of $\bmu$. This fact will be used to provide a uniform estimate in the proof of continuity.

    However, we emphasize that the local definition of $\cP_{M,\epsilon}$ seems to be indispensable for the invariance property in Theorem \ref{thm:pathwiseMFE} due to the polynomial dependence of $ \llbracket(\tilde \sigma^0,\tilde \sigma')\rrbracket_{\bB;\beta,\beta'}$ in the a priori estimate \eqref{eq:aprioriest}.
\end{remark}

\begin{lemma}
    For any $M,\epsilon>0$, $\cP_{M,\epsilon}$ is a convex and {compact} subset of $\cP(\cX)$, equipped with the topology of weak convergence.
\end{lemma}

\begin{proof}
    We first prove that it is convex. Suppose $\bmu^1,\bmu^2\in \cP_{M,\epsilon}$, with representations $(Y^1, Y^{1'})$, $(Y^2,Y^{2'})$ on $(\tilde \Omega^1,\tilde \mF^1,\tilde \F^1,\tilde \mP^1)$ and $(\tilde \Omega^2,\tilde \mF^2,\tilde \F^2,\tilde \mP^2)$ respectively. To construct a representation of $\lambda\bmu^1+(1-\lambda)\bmu^2$ for $\lambda\in (0,1)$, we consider a probability space $(\tilde \Omega^3,\tilde \F^3, \tilde \mP^3)$ supporting a Bernoulli random variable $\xi$ such that $\tilde \mP^3(\xi=1)=\lambda$. Then, define $\tilde \Omega = \tilde \Omega^1\times \tilde \Omega^2\times \tilde \Omega^3$, $\tilde \mF = \tilde \mF^1\otimes \tilde \mF^2$, $\tilde \F=\tilde \F^1\otimes \tilde\F^2\otimes \tilde \F^3$, $\tilde \mP=\tilde \mP^1\times \tilde \mP^2\times \tilde \mP^3$. All $\sigma$-fields, filtration, stochastic processes and random variables naturally extend to the large space $\tilde \Omega$. Moreover, define two processes $(Y,Y')$ on $\tilde \Omega$ as follows: $Y=Y^1$ and $Y'=Y^{1'}$ if $\xi=1$; $Y=Y^2$ and $Y'=Y^{2'}$ if $\xi=0$. By construction $\xi$ is independent of $\tilde \mF$, $(Y^j,Y^{j'})$ is independent of $\tilde \mF^{2-j}$ for $j=0,1$. We then compute the rough path norms of $(Y,Y')$. For any $(s.t)\in \triangle$ and $A\in \tilde \F_s$,
    \begin{align}
        \mE^{\tilde \mP}[|Y_t-Y_s|^m1_A]=&\mE^{\tilde \mP}[|Y_t-Y_s|^m1_A1_{\xi=1}]+\mE^{\tilde \mP}[|Y_t-Y_s|^m1_A1_{\xi=0}]\\
        =&\lambda\mE^{\tilde \mP}[|Y^1_t-Y^1_s|^m1_A]+(1-\lambda)\mE^{\tilde \mP}[|Y^2_t-Y^2_s|^m1_A]\\
        =&\lambda \mE^{\tilde \mP}\big[\mE^{\tilde\mP}[|Y^1_t-Y^1_s|^m|\tilde \F^1_s]1_A\big]+(1-\lambda) \mE^{\tilde \mP}\big[\mE^{\tilde\mP}[|Y^2_t-Y^2_s|^m|\tilde \F^2_s]1_A\big].
    \end{align}
    Therefore,
    \begin{align}
        \mE^{\tilde \mP}[|Y_t-Y_s|^m|\tilde \F_s]=&\lambda\mE^{\tilde\mP}[|Y^1_t-Y^1_s|^m|\tilde \F^1_s]+(1-\lambda)\mE^{\tilde\mP}[|Y^2_t-Y^2_s|^m|\tilde \F^2_s]\\
        \leq & \big(\lambda\|\delta Y^1\|^m_{\beta;m,\infty;[s,t];\tilde \mP}+(1-\lambda)\|\delta Y^2\|^m_{\beta;m,\infty;[s,t];\tilde \mP}\big)(t-s)^{m\beta},
    \end{align}
    yielding
    \begin{align}
        \|\delta Y\|^m_{\beta;m,\infty;[s,t];\tilde \mP}\leq \lambda\|\delta Y^1\|^m_{\beta;m,\infty;[s,t];\tilde \mP}+(1-\lambda)\|\delta Y^2\|^m_{\beta;m,\infty;[s,t];\tilde \mP}
    \end{align}
    Repeating the same argument to $Y'$ and $\mE_\cdot[R^Y]$, and adding all pieces together\footnote{Note that the uniform bound of $\|Y'\|_m$ is easier due to the convexity of $x\mapsto x^m$.}, we obtain, whenever $t-s<\epsilon$,
    \begin{align}
         \|(Y,Y')\|^m_{\bB;\beta,\beta';m,\infty;[s,t];\tilde \mP}\leq & \frac{\lambda}{3}(\|\delta Y^1\|^m_{\beta;m,\infty;[s,t];\tilde \mP}+\| Y^{1'}\|^m_{\beta';m,\infty;[s,t];\tilde \mP}+\|\mE_\cdot[R^{Y^1}]\|_{\beta+\beta';\infty;[s,t];\tilde \mP})\\
         &+\frac{1-\lambda}{3}(\|\delta Y^2\|^m_{\beta;m,\infty;[s,t];\tilde \mP}+\| Y^{2'}\|^m_{\beta';m,\infty;[s,t];\tilde \mP}+\|\mE_\cdot[R^{Y^2}]\|_{\beta+\beta';\infty;[s,t];\tilde \mP})\\
         \leq &M^m.
    \end{align}
    This proves $\lambda\bmu^1+(1-\lambda)\bmu^2\in \cP_{M,\epsilon}$.

    Next, we prove that it is compact. Since bounded rough path norms, together with integrability of the initial distribution $\lambda$, provide tightness (see the proof of Lemma \ref{lemma:tight}), it is sufficient to prove closedness. Suppose $\bmu^n\in \cP_{M,\epsilon}$, and $\bmu^n\to \bmu$ weakly in $\cP(\cX)$. Note that each $\bmu^n$ has a representation $(Y^n,Y^{n'})$ on $(\Omega^n,\mP^n)$. The family of joint laws $\{\cL^{\mP^n}(Y^n,Y^{n'})\}_{n\geq 1}$ is tight, because by the definition of $\cP_{M,\epsilon}$,
    \begin{align}
        &\mE^{\mP^n}[|Y^n_t-Y^n_s|^m]\lesssim (t-s)^{m\beta},\\
        &\mE^{\mP^{n}}[|Y^{n'}_t-Y^{n'}_s|^m]\lesssim (t-s)^{m\beta'}.
    \end{align}
    Therefore, applying Skorohod representation, we can find a common probability space $(\tilde \Omega',\tilde \mP')$ on which (up to a subsequence) we have almost sure convergence $(\tilde Y^n,\tilde Y^{n'})\to (\tilde Y,\tilde Y')$, and $\cL^{\tilde \mP'}(\tilde Y^n,\tilde Y^{n'})=\cL^{\mP^n}(Y^n,Y^{n'})$. By considering the filtration generated by $(\tilde Y,\tilde Y')$, the argument in the proof of Lemma \ref{lemma:pathwise:closedness} gives the desired rough path bound of $(\tilde Y,\tilde Y')$. Moreover, the uniqueness of weak convergence limit ensures that $\cL^{\tilde \mP'}(\tilde Y)=\bmu$. Therefore, $(\tilde Y,\tilde Y')$ is a representation of $\bmu$ satisfying the property in the definition of $\cP_{M,\epsilon}$, which proves $\bmu\in \cP_{M,\epsilon}$.
\end{proof}

\subsection{Step-3: the upper hemicontinuity of $\Phi$}

In this section, we establish the upper hemicontinuity of the fixed point mapping $\Phi$. The following uniqueness result of $R(\bB,\bmu)$ is useful.

\begin{proposition}\label{lemma:R:uniqueness}
    Fix indices $(\beta,\beta')\in \Pi$, $m\geq 2$. For any causal coupling $Q\in \cP_W(\cQ\times \cW)$, $\lambda \in \cP_m(\mR^d)$, $ \bmu\in \cL(\bD^{\beta,\beta'}_\bB L^{m,\infty})$, if $\mP^1,\mP^2\in R(\bB,\bmu)$ are such that $(\Lambda,W)_\#\mP^1 =(\Lambda,W)_\#\mP^2= Q$,  we have $\mP^1=\mP^2$.
\end{proposition}
\begin{proof}
The proof is mainly based on the RSDE characterization in Proposition \ref{prop:equiv} and the Yamada-Watanabe construction; see, e.g., Proposition 5.3.20 of \cite{karatzas1991brownian}. To be precise, first lift $\mP^j$, $j=1,2$, to the space $\cX\times \mR^d\times \cQ\times \cW$ to encode the initial distribution, i.e., $\hat\mP^j=(X_0,Y,\Lambda,W)_\#\mP^j$, $Y:=X-X_0$. Because $\cX$ and $\mR^d\times \cQ\times \cW$ are both Polish, there exist regular conditional probability distributions of $Y$ under $\mP^j$, given $ X_0=x_0,\Lambda=q, W=\bw$, denoted by $K^j(dy;x_0,q,\bw)$. On the extended probability space $\hat \Omega := \cX\times \cX\times \mR^d\times \cQ\times \cW$ with coordinate mapping $(\hat Y^1,\hat Y^2,\hat X_0,\hat\Lambda,\hat W)$, let us consider
\begin{align}
    \hat\mP(dy^1,dy^2,dx_0,dq,d\bw): = K^1(dy^1;x_0,q,\bw)\times K^2(dy^2;x_0,q,\bw)\lambda(dx_0)Q(dq,d\bw).
\end{align}
   Endow $\hat \Omega$ with $\hat\mF$, its natural filtration augmented and completed by $\hat\mP$. The standard argument of the Yamada-Watanabe theorem ensures that $\hat W$ is a $(\hat\mP,\hat \mF)$ Brownian motion, and, with $\hat X^j = Y^j+\hat X_0$, both $\hat X^1$ and $\hat X^2$ satisfy the RSDE \eqref{RSDE-for-PinR}. Assumption \ref{ass:} and the strong uniqueness in Theorem 4.6 of \cite{RSDE} yield $\hat X^1=\hat X^2$, $\hat \mP$-a.s.. Then $\mP^1=\cL^{\hat\mP}(\hat X_0+\hat Y^1,\hat\Lambda,\hat W)=\cL^{\hat\mP}(\hat X_0+\hat Y^2,\hat\Lambda,\hat W)=\mP^2$, which completes the proof.
\end{proof}

\begin{lemma}\label{lemma:Phi:continuity}
    The set-valued mapping $\Phi:\cP_{M,\epsilon}\ni \bmu\mapsto R^{\rm opt}(\bB,\bmu)$ is upper hemicontinuous.
\end{lemma}

\begin{proof}
   Suppose that $\bmu^n,\bmu\in \cP_{M,\epsilon}$, $\bmu^n\to\bmu$ weakly, and $\mP^n_*\in R^{\rm opt}(\bB,\bmu^n)$ and $\mP^n_*\to \mP_*$ weakly. We aim to prove that $\mP_*\in R^{\rm opt}(\bB,\bmu)$. By the property of $\cP_{M,\epsilon}$, we can choose a representation $(Y^n,Y^{n'})$ of $\bmu^n$ such that $\|(Y^n,Y^{n'})\|_{\bB;\beta,\beta',m;\infty}$ is uniformly bounded for $n\geq 1$. By Lemma \ref{lemma:construct_cvf} and the a priori estimate in Lemma \ref{lemma:tight}, $\|(X,\tilde\sigma^n(X)\|_{\bB;\beta,\beta';m,\infty}$ is uniformly bounded, where we denote $\tilde\sigma^n_t(x):=\tilde \sigma^0(t,x,\mu^n_t)$, $\tilde\sigma^{n'}_t(x) = \mE^{\tilde \mP'}[\partial_\mu \sigma^0(t,x,\mu^n_t)(Y^n)\cdot Y^{n'}]$, and $\hat\sigma^{n'}_t(x)=\nabla\tilde \sigma^n_t(x)\tilde \sigma^n_t(x)+\tilde\sigma^{n'}_t(x)$. Here, by $\bmu^n\to \bmu$ and the Skorokhod representation, we assume without loss of generality that the representation $(Y^n,Y^{n'})$ of $\bmu^n$ and $(Y,Y')$ of $\bmu$ are given on the same probability space $(\tilde \Omega',\tilde \mF',\tilde \F',\tilde \mP')$, and $\sup_{t\in [0,T]}\{|Y^n-Y|+|Y^{n'}-Y'|\}]\to 0$, $\tilde \mP'$-a.s.. In particular, because $\lambda\in \cP_m(\mR^d)$ and $m\beta>1$ in our setting, Kolmogorov's continuity theorem (Theorem 3.1 of \cite{rough_path_book}) yields $\mE[\sup_{t\in [0,T]}|Y^n_t|^m]\lesssim \mE[|Y_0|^m] + \mE[\sup_{t\in [0,T]}|\delta Y^n_{0,t}|^m]\lesssim \mE[|Y_0|^m] + \|(Y^n,Y^{n'})\|_{\bB;\beta,\beta',m,\infty}$, with a bound uniform in $n$. Therefore, we can upgrade the weak convergence $\bmu^n\to \bmu$ to $\cW_m$ convergence $\cW_m(\mu^n_t,\mu_t)\to 0$, $\forall t\in [0,T]$.

   As $\mP^n_*\to \mP_*$, we also have a common probability space $(\Omega',\F',\mP')$ such that $(\tilde X^n,\tilde \Lambda^n,\tilde W^n)\to (\tilde X,\tilde \Lambda,\tilde W)$, $\mP'$ almost surely; see the similar argument in the proof of Lemma \ref{lemma:pathwise:closedness}. As a consequence of Lemma \ref{lemma:T-phi-rough-path-norm}, we further have uniform boundedness of $\|(\Ti^n(\phi),\Ti^{n'}(\phi)\|_{\bB;\beta,\beta';m,\infty}$, where
   \begin{align}
        &\Ti^n_t(\phi) :=\nabla \phi(\tilde X^n_t) \tilde\sigma_t^n(\tilde X^n_t) \label{eq:Tn-def}\\
	&\Ti^{n'}_t(\phi) :=\nabla^2\phi(\tilde X^n_t)( \tilde\sigma^n_t(\tilde X^n_t),\tilde\sigma^n_t(\tilde X^n_t)) +\nabla\phi(\tilde X^n_t)\tilde\sigma^{n'}_t(\tilde X^n_t). \label{eq:Tnp-def}
   \end{align}
   Slightly abusing notation, we also define $(\Ti(\phi),\Ti'(\phi))$ similarly as in \eqref{eq:Tn-def}-\eqref{eq:Tnp-def}, with $\bmu^n$ replaced by $\bmu$ and $\tilde X^n$ replaced by $\tilde X$\footnote{This definition is consistent with the notation in \eqref{eq:Ti-def}-\eqref{eq:Tip-def}, though on different probability spaces.}. Thanks to the Lipschitz continuity of $\phi$, $\sigma^0$ and $\partial_\mu \sigma^0$, we have the following estimate for each $t\in [0,T]$:
   \begin{align}
    \mE^{\mP'}[|\Ti^n_t(\phi)-\Ti_t(\phi)|^m]+\mE^{\mP'}[|\Ti^{n'}_t(\phi)-\Ti'_t(\phi)|^m]\lesssim \|\nabla\phi\|_{C^2}(\mE^{\mP'}[|\tilde X^n_t-\tilde X_t|^m]+\cW_m(\mu^n_t,\mu_t)^m ).
   \end{align}
By virtue of the uniform boundedness of $\|(\Ti^n(\phi),\Ti^{n'}(\phi))\|_{\bB;\beta,\beta';m,\infty}$, we may invoke Lemma 4.20 of \cite{RSDE} to get
\begin{align}
    \sup_{t\in [0,T]}\bigg|\int_0^t (\Ti^n_s(\phi),\Ti^{n'}_s(\phi)d\bB-\int_0^t (\Ti_s(\phi),\Ti^{'}_s(\phi)d\bB\bigg|\to 0, {\rm\ \ in\ \ } L^m(\Omega',\mP').\label{convergence:rough-integral}
\end{align}
Repeating the arguments in the proof of Lemma \ref{lemma:pathwise:closedness}, we readily obtain $\mP_*\in R(\bB,\bmu)$. In particular, we have used the martingale property on both spaces $\Omega$ and $\Omega'$.

To prove that $\mP_*\in R^{\rm opt}(\bB,\bmu)$, we pick an arbitrary $\mP\in R(\bB,\bmu)$. By Proposition \ref{lemma:nonemptyness:strong}, we find a sequence of $\mP^n\in R(\bB,\bmu^n)$ such that $(\Lambda,W)_\# \mP^n\equiv (\Lambda,W)_\# \mP$. However, because $\{\bmu^n\}_{n\geq 1}\subset \cP_{M,\epsilon}$, we know that $\{\mP^n\}_{n\geq 1}$ is tight, and we find a limit $\mP^o\in R(\bB,\bmu)$ under weak convergence, with $(\Lambda,W)_\#\mP^o=(\Lambda,W)_\#\mP$. By Proposition \ref{lemma:R:uniqueness}, it holds that $\mP^o=\mP$. Letting $n\to \infty$ on both sides of $J_{\rm pw}(\mP^n;\bmu^n)\geq J_{\rm pw}(\mP^n_*;\bmu^n)$, we obtain $J_{\rm pw}(\mP;\bmu)\geq J_{\rm pw}(\mP_*;\bmu)$. Therefore, $\mP_*\in R^{\rm opt}(\bB,\bmu)$.
\end{proof}

\begin{remark}\label{rmk:hemicts}
    Because the cost function $J$ is continuous in $\mP$ under the weak topology, it is well known that, to establish the upper hemicontinuity of $\Phi$, we only need to show continuity of the {\it admissible-control set mapping} $\bmu\mapsto R(\bB,\bmu)$ (i.e., both upper and lower hemicontinuity). In \cite{Lacker2015}, this is achieved by using strong solvability and stability of SDEs, together with a Gronwall-type estimate. However, this approach creates technical challenges here. More precisely, stability of the RSDE solution maps is based on {\it rough path norms} (see Theorem 3.9 of \cite{RSDE}). As a consequence, if we wish to establish similar convergence via the RSDE strong solution maps as in \cite{Lacker2015}, we would need $\bmu^n\to\bmu$ to hold in a suitable rough path sense, whereas we work with the weak convergence topology.
\end{remark}

\subsection{Step-4: existence of fixed point}

\begin{theorem}\label{thm:pathwiseMFE}
    For any $m\geq 4$, there exists a pathwise MFE.
\end{theorem}
\begin{proof}
    Recall the definition of the cost functional $J_{\rm pw}$ in \eqref{cost-pathwise}. Because $f$ and $g$ are assumed to be bounded, it is clear that $\mP \mapsto J(\mP;\bmu)$ is continuous under the weak convergence topology. Therefore, $R^{\rm opt}(\bB,\bmu)$ is non-empty, compact, and convex. Consider the following set-valued mapping:
    \begin{align}\label{eq:def:Phi}
        \Phi:  \cP_{M,\epsilon}\ni \bmu \mapsto \{\bnu=X_\#\mP:\mP\in R^{\rm opt}\big(\omega^1,\lambda,\bmu\big)\}.
    \end{align}
By Lemma \ref{lemma:Phi:continuity}, it is upper hemicontinuous. Also, noting that the estimates in Lemma \ref{lemma:construct_cvf} and Lemma \ref{lemma:tight} are local for $s<t$, we apply these two lemmas to the measure flow $\bmu$ and any probability measure $\mP\in R(\bB,\bmu)$ to obtain the bound
\begin{align}
    \|(X,\tilde\sigma^0(X))\|_{\bB;\alpha,\beta;m,\infty;[s,t];\mP}\lesssim 1+M^{\gamma''}.
\end{align}
Since $\beta<\alpha$ by our choice, it holds that
\begin{align}\label{invariant:localestimate:nu}
     &\|(X,\tilde\sigma^0(X))\|_{\bB;\beta,\beta';m,\infty;[s,t];\mP}\lesssim 1+M^{\gamma''}\epsilon^{\delta'},\quad t-s<\epsilon
\end{align}
for some $\delta'\in (0,1)$\footnote{The constant 1 appears because the definition of $\|(X,\tilde \sigma^0(X))\|_{\bB;\beta,\beta';m,\infty}$ includes a term $\|\tilde\sigma^0(X)\|_\infty$, which is bounded by a universal constant $\|\tilde\sigma^0\|_\infty$.}. We emphasize that this holds for any $\mP\in R(\bB,\bmu)$, and the hidden constants depend on neither $M$ nor $\epsilon$. Therefore, if we choose $M$ and $\epsilon$ such that $1+M^{\gamma''}\epsilon^{\delta'} \lesssim M$, $(X,\tilde \sigma^0(X))$ can serve as a representation of $\bnu:=X_\#\mP\in\cL(\bD^{\beta,\beta'}_\bB L^{m,\infty})$, and by \eqref{invariant:localestimate:nu} we have $\bnu\in \cP_{M,\epsilon}$. Finally, Kakutani's fixed point theorem (c.f. \cite{Fan-Fixed-Point}) ensures the existence of a pathwise MFE.
\end{proof}

\section{Connections to MFG with Brownian Common Noise}\label{sec:randomization}

Although the main body of this paper studies MFG problems in which the common noise is modeled by a {\it deterministic} rough path, it is natural to investigate connections with the classical setting of {\it randomized} Brownian common noise. Indeed, a natural input for the rough path parameter $\bB$ is the path of a Brownian motion (see Chapter 3, \cite{rough_path_book}), which is itself an important motivation for rough path theory. Some recent studies have pointed out the relationship between RSDE theory (c.f. \cite{RSDE}) and ``doubly stochastic systems"; see, e.g., \cite{randomisation_RSDE} and \cite{controleedRSDE}). However, to our knowledge, there have been very few discussions in the context of mean field games. This section is devoted to establishing such connections.

We begin with notation.

\subsection{Preliminaries and more notations}

The {\it randomization} map is defined by
\begin{align}
   \fB:  \Omega^0\ni \omega^0  \mapsto \bB:=\bigg( B:=\bar B^0(\omega^0), \mB:=\bigg(\int \bar B^0 \otimes d \bar B^0\bigg) (\omega^0)\bigg)\in \hat\sC^{0,\alpha}.\end{align}
which is $\mF^{ B^0}$-adapted in the sense that $\omega^1\mapsto (B_t,\mB_{0,t})$ is $\fC_T/\F^{B^0}_t$-measurable. Here, $\fC_T$ is the Borel $\sigma$-field of $\hat\sC^{0,\alpha}$ generated by the metric $\rho_\alpha$.

In the case of randomized common noise, it is more convenient to consider the {\it joint law} of $(X,\Lambda,W)$ as a random element, not just the law of $X$. Therefore, in this section, we use $\mu$ and $\nu$ to denote probability measures on $\Omega$, and use $\mu^x$ to denote the $X$-marginal (see Definition \ref{def:defs-randomized-common-noise}-1). Consider the probability space of random inputs $\hat\Omega^0:=\hat\sC^{0,\alpha}\times \cP(\Omega)$ with canonical element $(\hat \bB,\hat\mu)$. With a slight abuse of notation, we also consider the probability space $\hat \Omega = \hat \sC^{0,\alpha}\times \cP(\Omega)\times \cP(\Omega)$ with canonical element $(\hat \bB,\hat \mu,\hat \nu)$\footnote{Keep in mind that $\hat \bB$ has two coordinates $\hat\bB = (\hat B,\hat \mB)$.}. Moreover, we lift the Wiener measure $\mP^0\in \cP(\Omega^0)$ to a canonical measure $\mP^0_{\bB}: = \fB_\# \mP^0$ on $\hat \sC^{0,\alpha}$, where we recall that $\fB:\Omega^0\to \hat \sC^{0,\alpha}$ is the randomization map of Brownian rough paths. By the standard construction of $\fB$, the enhanced Brownian motion is an almost surely defined measurable mapping, and its first level is the Brownian motion itself\footnote{For references, see, e.g., Section 3.2.1 of \cite{FrizVictoir2005BrownianRoughPath} and note that It\^o and Stratonovich enhancement only differ by a deterministic term.}. Hence, for every \(t\in[0,T]\), with $(B,\mB)=\fB(B)$,
\[
\sigma\bigl((B_s,\mathbb B_{0,s}):0\le s\le t\bigr)
=
\sigma(B_s:0\le s\le t),
\]
up to completion. Therefore, in what follows we will not distinguish between "conditioning on $\bB$" and "conditioning on $B$".

An important step in our reformulation is to observe that the usual compatibility condition can be reduced to the {\it causality} of certain couplings, which is now a distributional property. Indeed, causality of the couplings ensures that the Brownian motion property (in particular, independence of increments) can be preserved even when we enlarge the filtration. See Appendix \ref{app:causal-coupling} for a similar argument in the pathwise case, and the proof of Proposition \ref{prop:randomized-weak} for its application in the randomized setting. We thus use the following definition in this section:

\begin{definition}\label{def:defs-randomized-common-noise}
    \begin{enumerate}
        \item For $\mu\in \cP(\Omega)$ and $t\in [0,T]$, we consider $\mu^x := X_\#\mu$, $\mu^x_t = (X_t)_\#\mu$, and $\mu_{\cdot \wedge t}:= (X_{\cdot \wedge t}, \Lambda|_{[0,t]\times U}, W_{\cdot \wedge t})_\#\mu$.
        \item On $\hat \Omega^0$, the filtration generated by $\hat \mu$ is defined as follows:
        \begin{align}
        \F^{\hat \mu}_t=\sigma\big( \{\mE^{\hat \mu}[\phi(X_{\cdot \wedge t},\Lambda|_{[0,t]\times U},W_{\cdot \wedge t})]:\phi {\rm\ is\ bounded\ continuous}\}  \big).
        \end{align}
        We can similarly define the filtration generated by measure-valued processes on any probability space supporting such processes (such as $\hat\Omega$ and $\bar\Omega$).
        \item A probability measure $\hat Q$ on $\hat \Omega^0$ is said to be a causal coupling if for any $t\in [0,T]$, $\F^{\hat\mu}_t$ is conditionally independent of $\F^{\hat \bB}_T$, given $\F^{\hat \bB}_t$, under $\hat Q$. We denote $\hat Q\in \cP_\mc(\hat \Omega^0)$ if it is a causal coupling. Similarly, a probability measure $\hat \bP$ on $\hat\Omega$ is said to be a causal coupling, denoted by $\hat \bP\in \cP_\mc(\hat\Omega)$, if for any $t\in [0,T]$, $\F^{\hat\mu,\hat\nu}_t$ is conditionally independent of $\F^{\hat \bB}_T$, given $\F^{\hat \bB}_t$, under $\hat \bP$.
    \end{enumerate}
\end{definition}

For a $\hat\bP\in \cP(\hat \Omega)$, we can rewrite the cost functional as
\begin{align}
    \bJ(\hat\bP): =& \mE^{\bar\bP}\bigg[\mE^{\hat \nu}\bigg[\int_0^T f(t,X_t,\hat \mu^x_t,u)\Lambda_t(du)dt  + g(X_T,\hat \mu^x_T) \bigg]   \bigg]\\
    =&\int_{\hat \Omega}  J_{\rm pw}(\mu,\nu) \hat\bP(d\bB,d\mu,d\nu)
\end{align}
where $\hat \mu^x= X_\#\hat \mu$, and we recall that (see \eqref{cost-pathwise})
\begin{align}
    & J_{\rm pw}(\mu,\nu) := \int_\Omega \Gamma(x,q,\mu)\nu(d\omega),\\
    &\Gamma(x,q,\mu) = \int_0^T f(t,x_t,\mu^x_t,u) q(du,dt) + g(x_T,\mu^x_T).
\end{align}
We will also use the following two operators:
\begin{align}
   &T: \cP_{\mc}(\hat \Omega^0)\to \cP(\Omega\times \hat\Omega^0),\quad \hat Q\mapsto \mu(d\omega)\hat Q(d\bB,d\mu),\\
  & \hat T:\cP_\mc(\hat \Omega) \to \cP(\Omega\times \hat \Omega),\quad \hat \bP \mapsto \nu(d\omega)\hat \bP(d\bB,d\mu,d\nu),
\end{align}
and the following two special couplings on $\hat \Omega$ based on $\hat Q$:
\begin{align}
    &\hat\bP^{\hat Q}_{\rm diag}(d\bB,d\mu,d\nu)=\delta_{\mu}(d\nu)\hat Q(d\bB,d\mu),\\
    &\hat \bP^{\hat Q}_{\rm ind}(d\bB,d\mu,d\nu) = \hat Q_{\bB}(d\mu)\otimes \hat Q_\bB(d\nu) \bar\mP^0_\bB(d\bB).
\end{align}

\begin{remark}
    The consistency condition $\cL(\bar X,\bar \Lambda,\bar W|\bar \bB,\bar\mu)=\bar\mu$, a.s., is implicitly encoded in the definition of $T$. On the other hand, to accommodate the admissible deviations of the representative agent in the definition of MFG problems, we also need to consider the scenario in which the consistency condition is violated. This motivates the extension from $\hat\Omega^0$ to $\hat \Omega$ and the definition of $\hat T$.
\end{remark}

\subsection{A pathwise characterization of weak equilibrium}
The starting point of our discussion is an observation from Section 6.1 of \cite{Lacker-common-noise}: the law of a weak equilibrium is fully determined by the law of $(B,\mu)$. In our case, this is equivalent to specifying a distribution $\hat Q$ on $\hat \Omega^0$.

\begin{definition}\label{weak:MFG:canonical}
\begin{enumerate}
    \item The canonical space for a MFG with Brownian common noise is $(\bar\Omega,\bar\mF)$, where $\bar\Omega:=\Omega\times \hat \Omega^0$ with canonical random elements $(\bar X,\bar \Lambda,\bar W,\bar \bB=(\bar B,\bar\mB),\bar\mu)$ and canonical filtration $\bar\mF$ generated by $(\bar X,\bar\Lambda,\bar W,\bar \bB,\bar\mu)$. With a slight abuse of notations, we also consider the extended canonical space $\Omega\times \hat \Omega$, with canonical random elements $(\bar X,\bar W,\bar W,\bar \bB,\bar \mu,\bar\nu)$.
    \item A probability measure $\hat Q\in\cP(\hat \Omega^0)$ is said to be a weak equilibrium of the MFG with Brownian common noise if the probability measure $\bar\mP:=T(\hat Q)$ (which is on $\bar \Omega$) satisfies
    \begin{enumerate}
        \item $(\bar X_0)_\# \bar\mP=\lambda$;
        \item $\bar X_0, \bar W$ and $(\bar \bB,\bar\mu)$ are independent under $\mP$, and $(\bar W,\bar B)$ are $(\bar \mP,\bar\mF)$ Brownian motions;
        \item The following state equation holds $\bar\mP$-a.s.:
        \begin{align}\label{state-dynamic-common-noise}
              \bar  X_t =\bar  X_0+&\int_0^t \int_U b(s,\bar X_s,\bar \mu^x_s,u)\bar \Lambda_s(du)ds + \int_0^t \sigma(s,\tilde X_s,\tilde \mu^x_s) d \bar W_s \\
            +& \int_0^t \sigma^0(s,\bar X_s,\bar \mu^x_s)d \bar B_s,\quad \forall t\in [0,T].
        \end{align}
        \item For any probability measure $\bar\mP'\in \cP(\bar\Omega)$ satisfying (a)-(c) and such that, for any $t\in [0,T]$, $\F^\Lambda_t$ is conditionally independent of $\F^{\bar X_0,\bar \bB,\bar W,\bar\mu}_T$ given $\F^{\bar X_0,\bar \bB,\bar W,\bar\mu}_t$, we have $J(\bar\mP)\leq J(\bar \mP')$, where
        \begin{align}
        J(\bar\mP) := \mE^{\bar \mP}\bigg[\int_0^T f(t,\bar X_t,\bar\mu^x_t,u)\bar\Lambda_t(du) + g(\bar X_T,\bar \mu^x_T)\bigg]
        \end{align}
        is the aggregate cost functional.
    \end{enumerate}
\end{enumerate}
\end{definition}

In what follows, we prove that the best-response step of MFG problems with Brownian common noise can be recast as an optimization problem over the joint law of $(\hat \bB,\hat\mu,\hat\nu)$ on $\hat \Omega$, under which, almost surely, $\hat\nu$ solves the RSDE with input $(\hat \bB,\hat\mu)$. In this approach, the state process constraint is inherently expressed pathwise, and the use of RSDE theory therefore seems necessary.

\begin{definition}
    For a given $\hat Q \in \cP_{\rm c}(\hat \Omega^0)$, define
    \begin{align}
    \hat \cR(\hat Q):= \big\{\bar \bP \in \cP_{\rm c} (\hat \Omega): (\hat \bB,\hat \mu)_\#\hat\bP = \hat Q, \hat \bP(\hat \nu \in R(\hat\bB,\hat \mu,\lambda))=1\big\},
    \end{align}
    and
    \begin{align}
    \hat\cR^{\rm opt}(\hat Q): = \{\hat \bP_* \in \hat \cR
(\hat Q): \bJ(\hat \bP_*) = \inf_{\hat \bP\in \hat\cR(\hat Q)}\bJ(\hat \bP) \}.
\end{align}
\end{definition}

\begin{remark}
\begin{enumerate}
    \item The pathwise admissible set $R$ also provides a convenient description of mean-field control problems with common noise. Indeed, consider
    \begin{align}
        \hat \cR_{\rm MFC}:=\{\hat Q\in \cP_\mc(\hat \Omega): \hat Q(\hat \mu\in R(\hat \bB,\hat \mu))=1\}.
    \end{align}
    If we can prove $\hat \cR_{\rm MFC}$ is nonempty and compact, then the {\it weak} optimal relaxed control of MFC problems exists. We expect that a large part of the proof of this compactness result relies on our study of the pathwise admissible set $R$. As the focus of this paper is MFG problems, we leave the detailed study for future work. It is also worth noting that $\hat \cR_{\rm MFC}$ is exactly the set of self-consistent admissible laws in the MFG problems (i.e., presolutions), and $\hat \bP^{\hat Q}_{\rm diag}\in \hat \cR(\hat Q)$ if and only if $\hat Q\in \hat\cR_{\rm MFC}$.
\item Another interesting observation regarding MFC problems is that, in Section 3.4 of \cite{DjetePossamaiTan2022McKeanVlasovOptimalControl}, the authors observe (in a more general setting than ours, i.e., $\sigma$ is also controlled) that requiring martingale properties under a {\it pathwise} realization of $\hat \mu$ is essential for establishing approximation results for relaxed controls. It is natural to ask how the rough path (along with RSDE) formulation of MFC and MFG problems extends to the controlled $\sigma$ case and whether it is equivalent to the relaxed formulation in \cite{DjetePossamaiTan2022McKeanVlasovOptimalControl}. We aim to explore this in future work.
\end{enumerate}

\end{remark}

The following randomization lemma is important for the subsequent analysis.

\begin{lemma}\label{lemma:randomization}
    Suppose $\hat Q\in \hat \cR_{\rm MFC}$, $\bar\mP\in \cP(\bar\Omega)$ and $(\bar\bB,\bar\mu)_\#\bar\mP=\hat Q$. Denote $$\hat\bP = \cL^{\bar\mP}\big(\bar\bB,\bar\mu,\cL^{\bar\mP}(\bar X,\bar\Lambda,\bar W|\bar\bB,\bar\mu)\big). $$
    Then, the state equation \eqref{state-dynamic-common-noise} holds $\bar\mP$-a.s. with $\bar X_0\sim \lambda$ if and only if $\hat\bP(\hat\nu\in R(\hat\bB,\hat\mu,\lambda))=1$, i.e., if and only if $\hat\bP\in \hat\cR(\hat Q)$.
\end{lemma}

\newcommand{\bbarP}{\overline{\overline{\mP}}}

\begin{proof}
    {\it Sufficiency.} For convenience, we denote $\bbarP:=\hat T(\hat\bP)$ so that $\bar\mP=\cL^{\bbarP}(\bar X,\bar\Lambda,\bar W,\bar\bB,\bar\mu)$. Suppose $\hat\bP(\hat\nu\in R(\hat\bB,\hat\mu,\lambda))=\bbarP(\bar\nu\in R(\bar \bB,\bar\mu))=1$. Recalling Lemma \ref{prop:equiv} and invoking the rough It\^o formula for the triplet $(\bar X,\bar W,\bar \bB)$, we know that for $\bbarP$-a.s. $(\bar \bB,\bar\mu,\bar\nu)$ and any test function $\phi\in C_0^\infty(\mR^{d+n+l})$,
    \begin{align}\label{martingale:XWB}
     M^{\bar\bB,\bar\mu}_t(\phi):= &\phi(X_t,W_t,\bar B_t)-\phi(X_0,W_0,\bar B_0)-\int_0^t \int_U \overline \mL\phi(s,X_s,W_s,\bar B_s,\bar\mu_s,u)\bar\Lambda_s(du)\\
     &-\int_0^t \big(\bar \Ti_s(\phi),\bar\Ti'_s(\phi)\big) d\bar \bB_s, \quad \forall t\in [0,T]
    \end{align}
is a $(\bar \nu,\mF)$-martingale (defined on $\Omega$), with the extended generator
\begin{align}
         \overline\mL \phi(t,x,w,\mb,\mu,u):=& \overline b(t,x,\mu,u)^\mt \nabla \phi(x,w,\mb)+ \frac{1}{2}\Tr\bigg(\overline a(t,x,\mu)\nabla^2\phi(x,w,\mb)\bigg)
    \end{align}
    and
    \begin{align}
        \overline b(t,x,\mu,u)=\left(\begin{array}{c}
             b(t,x,\mu,u)  \\
              0 \\
              0
        \end{array}\right), \quad \overline a(t,x,\mu) = \left(\begin{array}{cc}
             \sigma(t,x,\mu) & \sigma^0(t,x,\mu) \\
              I_{n\times n} & 0\\
              0 & I_{l\times l}
        \end{array}   \right)\left(\begin{array}{cc}
             \sigma(t,x,\mu) & \sigma^0(t,x,\mu) \\
              I_{n\times n} & 0\\
              0 & I_{l\times l}
        \end{array}   \right)^\mt.
    \end{align}
Here, similarly to \eqref{eq:Ti-def}-\eqref{eq:Tip-def}, we denote
\begin{align}
    &\bar \Ti_t(\phi) :=\nabla_x \phi_t\tilde \sigma_t^0(X_t)+\nabla_b\phi_t, \\
	&\Ti'_t(\phi) :=\nabla^2_x\phi_t(\tilde \sigma^0_t(X_t),\tilde\sigma^0_t(X_t))+\Tr(\nabla^2_b\phi_t) +\nabla_{xb}\phi_t(\tilde \sigma^0_t(X_t),I)+\nabla_{bx}(I,\tilde \sigma^0_t(X_t))+\nabla_x\phi_t\hat\sigma'_t(X_t),
\end{align}
where, for any differential operator $D$, we abbreviate $D\phi_t:=D\phi(X_t,W_t,\bar B_t)$ and replace the input flow $\{\mu_t\}_{t\in [0,T]}$ with $\{\bar \mu^x_t\}_{t\in [0,T]}$. Define a process $\bar M_t(\bar\omega) := M^{\bar\bB(\bar\omega),\bar\mu(\bar\omega)}_t(\omega)$. We claim that $\bar M$ is a $(\bar\mP,\bar\mF)$ martingale. Indeed, for any $t>s$ and bounded continuous $\psi$,
\begin{align}
\mE^{\bar\mP}[\bar M_t(\phi)\psi(\bar X_{\cdot\wedge s},\bar \Lambda|_{[0,s]\times U},\bar W_{\cdot \wedge s},\bar \bB_{\cdot\wedge s},\bar\mu_{\cdot\wedge s})]=&\mE^{\bbarP}[\bar M_t(\phi)\psi(\bar X_{\cdot\wedge s},\bar \Lambda|_{[0,s]\times U},\bar W_{\cdot \wedge s},\bar \bB_{\cdot\wedge s},\bar\mu_{\cdot\wedge s})]\\
=&\mE^{\hat\bP}\big[ \mE^{\hat \nu}[M^{\hat \bB,\hat\mu}_t(\phi)\psi(X_{\cdot\wedge s},\Lambda|_{[0,s]\times U},W_{\cdot \wedge s},\hat \bB_{\cdot\wedge s},\hat\mu_{\cdot\wedge s})]     \big]\\
=&\mE^{\hat\bP}\big[ \mE^{\hat \nu}[M^{\hat \bB,\hat\mu}_s(\phi)\psi(X_{\cdot\wedge s},\Lambda|_{[0,s]\times U},W_{\cdot \wedge s},\hat \bB_{\cdot\wedge s},\hat\mu_{\cdot\wedge s})]     \big]\\
=&\mE^{\bar\mP}[\bar M_s(\phi)\psi(\bar X_{\cdot\wedge s},\bar \Lambda|_{[0,s]\times U},\bar W_{\cdot \wedge s},\bar \bB_{\cdot\wedge s},\bar\mu_{\cdot\wedge s})].
\end{align}
On the other hand, because $\bar B$ is a $(\bar\mP,\bar\mF)$ Brownian motion (see the proof of Proposition \ref{prop:randomized-weak} below), the rough integral in \eqref{martingale:XWB} becomes an It\^o integral under $\bar\mP$, hence also a $(\bar\mP,\bar\mF)$ martingale. Therefore, we conclude that
\begin{align}
    \bar N_t: = \phi(\bar X_t,\bar W_t,\bar B_t)-\phi(\bar X_0,\bar W_0,\bar B_0)-\int_0^t \int_U \overline \mL\phi(s,\bar X_s,\bar W_s,\bar B_s,\bar\mu_s,u)\bar\Lambda_s(du)
\end{align}
is a $(\bar\mP,\bar\mF)$ martingale. By a standard result, this verifies the state equation \eqref{state-dynamic-common-noise}.

    {\it Necessity.} Suppose that the state equation \eqref{state-dynamic-common-noise} holds $\bar\mP$-a.s.. We first observe that $\bar\mP(\bar\mu^x \in \cL(\bD_{\bar \bB}^{\beta,\beta'}L^{m,\infty}))=\hat Q(\hat\mu^x \in \cL(\bD_{\hat \bB}^{\beta,\beta'}L^{m,\infty}))=1$, by the definition of $\hat \cR_{\rm MFC}$. On the other hand, for {\it any} $(\bB,\mu,\nu)\in \hat \sC^{0,\alpha}\times \cP(\Omega)\times \cP(\Omega)$ such that $\mu^x \in \cL(\bD_{ \bB}^{\beta,\beta'}L^{m,\infty})$, construct a probability space $(\tilde \Omega,\tilde \mP)$ with random variables $(\tilde \xi, \tilde \Lambda,\tilde W)$ such that $\cL^{\tilde \mP}(\tilde\xi, \tilde \Lambda, \tilde W) = \cL^{\nu}(X_0,\Lambda, W)$. Consider the following RSDE:
    \begin{align}
\begin{cases}
d \tilde X_t = & \bar b(t,\tilde X_t,\mu^x_t,\tilde\omega)dt +  \sigma(t,\tilde X_t,\mu^x_t)d\tilde W_t\\
&+\big(\tilde \sigma^0_t(\tilde X_t), \tilde \sigma'_t(\tilde X_t)\big)d\bB,\\
\tilde X_0 =& \tilde \xi.
\end{cases}
\label{eq:mkv-RSDE}
\end{align}
Under Assumption \ref{ass:}, there exists a unique $L_{m,\infty}$ solution $\tilde X^{\bB,\mu,\nu}$ to \eqref{eq:mkv-RSDE}. Moreover, by Theorem 4.3 of \cite{randomisation_RSDE} and the discussion thereafter, we can choose a progressively measurable version $(\bB,\mu,\nu)\mapsto \cL^{\tilde \mP}(\tilde X^{\bB,\mu,\nu})$, and if $(\hat \bB,\hat\mu,\hat\nu)$ is a random variable with distribution $\hat\bP$, $\tilde X^{\hat \bB,\hat\mu,\hat\nu}$ solves the following SDE $\tilde\mP$-a.s. (we still denote by $\tilde \mP$ the extended probability measure, obtained by augmenting $\tilde \mP$ with the randomness of $(\hat\bB,\hat\mu,\hat\nu)$):
    \begin{align}
\begin{cases}
d \tilde X_t =& \bar b(t,\tilde X_t,\hat\mu^x,\tilde\omega)dt +  \sigma(t,\tilde X_t,\hat\mu^x)d\tilde W_t\\
&+ \sigma^0(t,\tilde X_t,\hat\mu^x)d\tilde B_t,\\
\tilde X_0 =&\tilde \xi.
\end{cases}
\label{eq:mkv-common-noise}
\end{align}
Moreover, for any bounded continuous function $\phi$ defined on $\mR^d\times \cQ\times \cW $, we have
\begin{align}
\mE^{\tilde \mP}[\phi(\tilde X_0,\tilde \Lambda,\tilde W)] = &\mE^{\hat \bP}\big[\mE^{\hat \nu}[\phi(X_0,\Lambda,W)]\big]\\
=&\mE^{\hat Q}[\mE^{\bar\mP}[\phi(\bar X_0,\bar\Lambda,\bar W)|\hat \bB,\hat \mu]]\\
=&\mE^{\bar\mP}[\phi(\barX_0,\bar\Lambda,\bar W)],
\end{align}
thus $\cL^{\tilde \mP}(\tilde X_0,\tilde \Lambda,\tilde W) = \cL^{\bar\mP}(\bar X_0,\bar\Lambda,\bar W)$. Denoting $\bar \mP':= \cL^{\tilde \mP}(\tilde X,\tilde \Lambda,\tilde W,\hat \bB,\hat \mu)$, by a similar Yamada-Watanabe type argument in Proposition \ref{lemma:R:uniqueness}, the strong uniqueness of SDE \eqref{eq:mkv-common-noise} (with random coefficients) gives $\bar \mP'=\bar\mP$. The uniqueness of the conditional kernel then implies that, for some $\Theta$ with $\mP^0(\Theta)=1$, it holds that $\cL^{\bar\mP}(\bar X,\bar \Lambda,\bar W|\bar\bB,\bar\mu)$ is the law of solution to \eqref{eq:mkv-RSDE}, $\forall \bar\omega\in \Theta$. The desired result then follows from the definition of $\hat\bP$.
\end{proof}

We now provide an equivalent characterization of weak MFG equilibrium. It crucially relies on our pathwise formulation and seems to be new for the problems with common noise.

\begin{proposition}\label{prop:randomized-weak}
    $\hat Q\in \hat \cR_{\rm MFC}$ is a weak equilibrium of the MFG with Brownian common noise if and only if $\hat \bP^{\hat Q}_{\rm diag}\in \hat \cR^{\rm opt}(\hat Q)$.
\end{proposition}
\begin{proof}
{\it Sufficiency.} For simplicity, in this proof let us denote $\hat\bP_* = \hat\bP^{\hat Q}_{\rm diag}$ and $\bar\mP = T(\hat Q)$. To verify conditions in Definition \ref{weak:MFG:canonical}, we will repeatedly use the consistency condition $\cL^{\bar \mP}(\bar X,\bar \Lambda,\bar W|\bar \bB,\bar\mu) = \bar\mu$. First, for any $A\in \B(\mR^d)$, we have $\bar\mP(\bar X_0\in A) = \mE^{\hat Q}[\bar\mP_{\bar\bB,\bar\mu}(\bar X_0\in A)] = \mE^{\hat Q}[\hat\mu(\bar X_0\in A)]=\lambda (A)$, because under $\hat Q$, $\hat\mu\in R(\hat\bB,\hat \mu,\lambda)$, hence $(X_0)_\#\hat \mu = \lambda$, holds almost surely. Similarly, the distribution of $\bar W$ under $\bar\mP$ is $\mP_W$, the Wiener measure on $\cW(=C([0,T];\mR^n))$. To verify the independence condition in (b), take arbitrary bounded continuous functions $\phi_\bB,\phi_{X_0}$ and $\phi_W$, defined on appropriate spaces respectively, and note that
\begin{align}
   \mE^{\bar\mP}[\phi_\bB(\bar\bB)\phi_{X_0}(\bar X_0)\phi_W(\bar W)]=& \mE^{\hat Q}\big[\phi_\bB(\hat \bB)\mE^{\hat \mu}[\phi_{X_0}(X_0)]\mE^{\hat \mu}[\phi_W(W)]\big] \\
   =&\mE^{\hat Q}[\phi_\bB(\hat \bB)]\mE^\lambda[\phi_{X_0}(X_0)]\mE^{\mP_W}[\phi_W(W)]\\
   =&\mE^{\bar \mP}[\phi_\bB(\bar\bB)]\mE^{\bar\mP}[\phi_{X_0}(\bar X_0)]\mE^{\bar\mP}[\phi_W(\bar W)].
\end{align}
This is the independence in condition (b) of Definition \ref{weak:MFG:canonical}. Next, to verify that $(\bar W, \bar B)$ are Brownian motions under $(\bar\mP,\bar\mF)$, we only need independence of increments. For $\bar W$, take bounded continuous functions $\phi$, $\psi_1$, $\psi_2$, and use the conditioning argument as before to conclude that, for any $t>s$,
\begin{align}
&\mE^{\bar\mP}[\phi(\bar W_t-\bar W_s)\psi_1(\bar X_{\cdot \wedge s},\bar\Lambda|_{[0,s]\times U})\psi_2(\bar \bB_{\cdot \wedge s},\bar \mu_{\cdot \wedge s} )]\\=&\mE^{\hat Q}\big[\mE^{\hat \mu}[\phi(W_t-W_s)]\cdot \mE^{\hat \mu}[\psi_1(X_{\cdot \wedge s},\Lambda|_{[0,s]\times U})]\psi_2(\hat \bB_{\cdot \wedge s},\hat \mu_{\cdot \wedge s} )\big]\\
=&\mE^{\mP_W}[\phi(W_t-W_s)]\mE^{\bar\mP}[\psi_1(X_{\cdot \wedge s},\Lambda|_{[0,s]\times U})\psi_2(\hat \bB_{\cdot \wedge s},\hat \mu_{\cdot \wedge s} )]\\
=&\mE^{\bar\mP}[\phi(W_t-W_s)]\mE^{\bar\mP}[\psi_1(\bar X_{\cdot \wedge s},\bar\Lambda|_{[0,s]\times U})\psi_2(\bar \bB_{\cdot \wedge s},\bar \mu_{\cdot \wedge s} )].
\end{align}
Thus $\bar W_t-\bar W_s$ is independent of $\bar\F_s$ under $\bar\mP$. To prove the independence of increment of $\bar B$, we observe that, for any bounded continuous $\psi$ and any $s\in [0,T]$, $\mE^{\hat \mu}[\psi(X_{\cdot \wedge s},\Lambda|_{[0,s]\times U})]$ is $\F^{\hat\mu}_s$ measurable. Therefore, because $\hat Q$ is a causal coupling between $\hat \bB$ and $\hat \mu$, we have, for bounded continuous $\phi$, $\psi_1$, $\psi_2$,
\begin{align}
    &\mE^{\bar\mP}[\phi(\bar B_t-\bar B_s)\psi_1(\bar X_{\cdot \wedge s},\bar\Lambda|_{[0,s]\times U},\bar W_{\cdot \wedge s})\psi_2(\bar \mu_{\cdot \wedge s} )]\\=&\mE^{\hat Q}\big[\phi(\hat B_t-\hat B_s) \mE^{\hat \mu}[\psi_1(X_{\cdot \wedge s},\Lambda|_{[0,s]\times U},W_{\cdot \wedge s})]\psi_2(\hat \mu_{\cdot \wedge s} )\big]\\
=&\mE^{\hat Q}\Big[\mE^{\hat Q}[\phi (\hat B_t-\hat B_s)|\F^{\hat \bB}_s ]\cdot\mE^{\hat Q}\big[\mE^{\hat \mu}[\psi_1(X_{\cdot \wedge s},\Lambda|_{[0,s]\times U},W_{\cdot \wedge s})]\psi_2(\hat \mu_{\cdot \wedge s} )\big|\F^{\hat \bB}_s\big]\Big]\\
=&\mE^{\bar\mP}[\phi(\bar B_t-\bar B_s)]\mE^{\bar\mP}[\psi_1(\bar X_{\cdot \wedge s},\bar \Lambda|_{[0,s]\times U},\bar W_{\cdot \wedge s})\psi_2(\bar \mu_{\cdot \wedge s} )].
\end{align}
Therefore $\bar B_t-\bar B_s$ is independent of $\bar\F_s$ under $\bar\mP$, completing the proof of condition (b) in Definition \ref{weak:MFG:canonical}. Condition (c) readily follows from the sufficiency part of Lemma \ref{lemma:randomization}, considering $\hat\bP = \hat \bP^{\hat Q}_{\rm diag}$.

We now verify condition (d). Indeed, $\tilde \mP\circ(\tilde X,\tilde\Lambda,\tilde W,\tilde \bB,\tilde\mu)^{-1} = \mu(d\omega)\hat Q(d\bB,d\mu)$, which in particular implies $\cL^{\tilde\mP}(\tilde X,\tilde \Lambda,\tilde W|\tilde \bB,\tilde \mu)=  \tilde \mu$, $\tilde \mP$-a.s.. Constraining to the space $\bar\Omega$, we obtain the consistency condition $\bar \mu = \cL^{\bar \mP}(\bar X,\bar  \Lambda,\bar W| \bar B, \bar \mu)$. Finally, for an arbitrary $\bar\mP'\in \hat \cR(\hat Q)$, we conclude from the necessity part of Lemma \ref{lemma:randomization} that $$\bar \bP' := \bar\mP'\circ(\bar X,\bar \Lambda,\bar W,\bar\bB,\bar\mu,\cL^{\bar\mP'}(\bar X,\bar \Lambda,\bar W|\bar \bB,\bar\mu))^{-1}\in \hat \cR(\hat Q).$$ Therefore, by the optimality of $\bar \bP_*$, we have $J(\bar \mP') = \bJ(\bar \bP')\geq \bJ(\bar\bP_*)=J(\bar\mP)$, yielding the optimality in the definition of weak equilibrium.

    {\it Necessity.} Suppose $\hat Q\in \hat \cR_{\rm MFC}$, hence $\hat \bP^{\hat Q}_{\rm diag}\in \hat \cR(\hat Q)$. For any $\hat \bP\in \hat \cR(\hat Q)$, consider $\hat \bP' = \cL^{\hat \bP}(\hat \bB,\hat \mu,\mE^{\hat\bP}[\hat \nu|\hat \bB,\hat \mu])$. The convexity of the pathwise admissible set $R(\bB,\mu)$ implies that $\hat \bP'\in \hat \cR(\hat Q)$. On the other hand, denoting $\bar\mP=\hat T(\hat\bP)\circ (\bar X,\bar \Lambda,\bar W,\bar\bB,\bar\mu)^{-1}$, we may verify $\hat \bP' = \cL^{\bar\mP}(\bar\bB,\bar\mu,\cL^{\bar\mP}(\bar X,\bar \Lambda,\bar W|\bar\bB,\bar\mu))$ as follows: for any bounded continuous function $\phi$ defined on $\hat \sC^{0,\alpha}\times \cP(\Omega)\times \cP(\Omega)$ of the type $\phi(\bB,\mu,\nu) = \phi_0(\bB,\mu,\omega)\nu(d\omega)$, we have
    \begin{align}
    \mE^{\hat \bP'}[\phi(\hat \bB,\hat\mu,\hat\nu)]=&\mE^{\hat\bP}\big[\phi(\hat \bB,\hat\mu,\mE^{\hat\bP}[\hat\nu|\hat\bB,\hat\mu])\big]   \\
    =&\mE^{\hat Q}\bigg[\mE^{\hat\bP}\bigg[\int_\Omega \phi_0(\bB,\mu,\omega)\hat\nu(d\omega)\bigg|\hat \bB,\hat\mu\bigg]\bigg]\\
    =&\mE^{\hat \bP}\bigg[\int_\Omega\phi_0(\hat\bB,\hat \mu,\omega)\hat\nu(d\omega)\bigg]\\
    =&\mE^{\hat T(\hat \bP)}[\phi_0(\bar\bB,\bar\mu,\bar X,\bar\Lambda,\bar W)]\\
    =&\mE^{\bar\mP}[\phi_0(\bar\bB,\bar\mu,\bar X,\bar\Lambda,\bar W)]\\
    =&\mE^{\bar\mP}\big[\phi\big(\bar\bB,\bar\mu,\cL^{\bar\mP}(\bar X,\bar\Lambda,\bar W|\bar\bB,\bar\mu)\big)\big].
    \end{align}
    Therefore, by the sufficiency part of Lemma \ref{lemma:randomization}, $\bar\mP$ solves the state equation \eqref{state-dynamic-common-noise}. By a similar argument regarding causal coupling in the sufficiency part of this proof, $\bar\mP$ also verifies condition (b) of Definition \ref{weak:MFG:canonical}-2. Denoting $\bar\mP_* = T(\hat Q)$, the optimality in the weak equilibrium condition implies that $J(\bar\mP_*)\leq J(\bar\mP)$. Because $J_{\rm pw}$ is linear in $\nu$, it is not hard to verify $\bJ(\hat\bP)=\bJ(\hat\bP')$, and consequently,
    \begin{align}
    \bJ(\hat\bP^{\hat Q}_{\rm diag})=J(\bar\mP_*)\leq J(\bar\mP) = \bJ(\hat \bP')=\bJ(\hat\bP).
    \end{align}
    Since $\hat\bP$ was arbitrary, $\hat\bP^{\hat Q}_{\rm diag}\in \hat\cR^{\rm opt}(\hat Q)$, completing the proof.
\end{proof}

\subsection{An existence proof of strong equilibrium}

Using the pathwise characterization from the previous subsection, we give an alternative proof of the existence of strong equilibrium under an additional structural assumption (see Assumption \ref{ass:strong}). We establish this result without requiring uniqueness of the best response under {\it any} environment (see Assumption U.4 of \cite{Lacker-common-noise}), and without first establishing pathwise uniqueness of MFG equilibrium. The price to pay is that we assume a strict version of the Lasry-Lions monotonicity condition, in which equality identifies the relevant measures.
\begin{assumption}\label{ass:strong}
    \begin{enumerate}
        \item The state dynamics do not have mean-field interaction, i.e., $b$, $\sigma$ and $\sigma^0$ do not depend on $\mu$;
        \item For any $\mu,\nu \in \cP(\Omega)$, we have
        \begin{align}\label{LL-monotonicity}
            \int_\Omega \big(\Gamma(x,q,\mu)-\Gamma(x,q,\nu)\big)(\mu-\nu)(d\omega) \geq 0,
        \end{align}
        and the equality in \eqref{LL-monotonicity} implies $\mu^x = \nu^x$.
    \end{enumerate}
\end{assumption}

\begin{proposition}
    Under Assumption \ref{ass:strong}, any weak equilibrium is in fact strong.
\end{proposition}

\begin{proof}
    Under Assumption \ref{ass:strong}-1, the definition of the pathwise admissible set $R$ does not depend on $\mu$, so we simply write $R(\bB)=R(\bB,\mu)$ for any $\mu$. Pick a weak equilibrium $\hat Q$. By Proposition \ref{prop:randomized-weak}, $\hat\bP^{\hat Q}_{\rm diag}\in \hat \cR^{\rm opt}(\hat Q)$. In particular, $\hat \bP^{\hat Q}_{\rm diag}\in \hat \cR(\hat Q)$, hence $\hat\bP(\hat\nu\in R(\hat \bB))=\hat Q(\hat\mu\in R(\hat\bB))=1$. Therefore $\hat\bP^{\hat Q}_{\rm ind}(\hat\nu \in R(\hat\bB))=\hat Q(\hat\mu\in R(\hat \bB))=1$, and $\hat\bP^{\hat Q}_{\rm ind}\in \hat \cR(\hat Q)$, implying $\bJ(\hat \bP^{\hat Q}_{\rm ind})\geq \bJ(\hat \bP^{\hat Q}_{\rm diag})$. On the other hand, by Assumption \ref{ass:strong}-2, for $\hat \bP^{\hat Q}_{\rm ind}$-a.s. $\hat\mu$ and $\hat \nu$,
    \begin{align}\label{as:LL}
      \int_\Omega \Gamma(x,q,\hat\mu)\hat\nu(d\omega)+\int_\Omega\Gamma(x,q,\hat\nu)\hat\mu(d\omega)\leq \int_\Omega \Gamma(x,q,\hat\mu)\hat\mu(d\omega)+\int_\Omega\Gamma(x,q,\hat\nu)\hat\nu(d\omega) .
    \end{align}
    Taking expectation with respect to $\hat\bP^{\hat Q}_{\rm ind}$, and noting that
    \begin{align}
    \mE^{\hat \bP^{\hat Q}_{\rm ind}}\bigg[\int_\Omega \Gamma(x,q,\hat\mu)\hat\nu(d\omega)\bigg]=&   \mE^{\mP^0_\bB}\bigg[\int_\Omega\mE^{\hat Q_\bB}[\Gamma(x,q,\hat\mu)]\mE^{\hat Q_\bB}[\hat\mu](d\omega)\bigg]\\
    =&\mE^{\hat \bP^{\hat Q}_{\rm ind}}\bigg[\int_\Omega \Gamma(x,q,\hat\nu)\hat\mu(d\omega)\bigg],
    \end{align}
we have $\bJ(\hat\bP^{\hat Q}_{\rm ind})\leq \bJ(\hat \bP^{\hat Q}_{\rm diag})$. By the optimality of $\hat\bP^{\hat Q}_{\rm diag}$, we conclude that $\bJ(\hat\bP^{\hat Q}_{\rm ind})= \bJ(\hat \bP^{\hat Q}_{\rm diag})$ and \eqref{as:LL} holds as equality for $\hat \bP^{\hat Q}_{\rm ind}$-a.s. $\hat\mu$ and $\hat\nu$. Assumption \ref{ass:strong}-2 then gives $\hat \mu^x = \hat \nu^x$, a.s. But under $\hat \bP^{\hat Q}_{\rm ind}$ they are conditionally independent given $\bB$. By a classical result (see, e.g., Theorem 8.11, \cite{kallenberg2021foundations}) $\hat\mu^x$ is measurable with respect to $\overline{\sigma(\bB)}\subset \overline{\sigma(B)}$. Here, the overline denotes completion under $\hat \bP^{\hat Q}_{\rm ind}$. Translating this condition to $\bar \mP:=T(\hat Q)$, we have $\bar\mu^x = \cL^{\bar\mP}(\bar X|\bar B,\bar \mu) = \cL^{\bar\mP}(\bar X|\bar B) $ a.s., which is precisely the consistency condition in strong equilibrium.
\end{proof}

\begin{remark}
   All discussions in this section potentially apply to other common noise models beyond Brownian motion. More specifically, we may replace the probability measure $\mP^0_\bB$ (currently the It\^o-lifted Wiener measure) with laws of other stochastic processes and change the rough path space when needed. Provided that the pathwise rough path integrals have natural connections with the corresponding probabilistic integrals, we can similarly characterize MFG problems with random common noise through the pathwise formulation. One example in this direction is fractional Brownian motion with Hurst index $H\in (1/3,1/2)$, where the probabilistic integrals are interpreted as Skorohod integrals.
\end{remark}

\appendix

\section{Causal couplings between $\Lambda$ and $W$}
\label{app:causal-coupling}

In this appendix, we present two auxiliary results regarding causal coupling, a special type of coupling between $\Lambda$ and $W$ that plays a core role in our proof of existence and uniqueness results for $R(\bB,\lambda,\bmu)$.

\begin{lemma}
    For any input $(\bB,\lambda,\bmu)$ and any $\mP\in R(\bB,\lambda,\bmu)$, $(\Lambda,W)_\#\mP\in \cP_W(\cQ\times \cW)$ and is a causal coupling.
\end{lemma}

\begin{proof}
    Denote $Q=(\Lambda,W)_\#\mP$. By Remark \ref{martignale-condition-special-case}, $W$ is a $(\mP,\mF)$ Brownian motion, hence $W_\#=\mP_W$. Moreover, $W-W_{\cdot\wedge t}$ is independent of $\F^{\Lambda,W}_t$ under $Q$. Hence, for any bounded measurable function $g$ on $\cW$, it holds that
    \begin{align}
        \mE^Q[g(W)|\F^{\Lambda,W}_t]=&\mE^Q[g(W_{\cdot \wedge t}+W-W_{\cdot \wedge t}|\F^{\Lambda,W}_t]\\
    =&\mE^Q[g(w+W-W_{\cdot \wedge t})]\bigg|_{w=W_{\cdot \wedge t}}\\
    =&\mE^{\mP_W}[g(w+W-W_{\cdot \wedge t})]\bigg|_{w=W_{\cdot\wedge t}}\\
    =:& \tilde g(W_{\cdot \wedge t}).
    \end{align}
    Noting that the right-hand side is $\F^W_t$-measurable, we have
    \begin{align}
        \mE^Q[g(W)|\F^{W}_t] = \tilde g(W_{\cdot \wedge t}).
    \end{align}
    On the other hand, for any bounded measurable function $h$ on $\cQ$ only depending on $\Lambda|_{[0,t]\times U}$, we have
    \begin{align}
        \mE^Q[h(\Lambda|_{[0,t]\times U})g(W)|\F^W_t]=&\mE^Q\big[h(\Lambda|_{[0,t]\times U})\mE^Q[g(W)|\F^{\Lambda,W}_t]\big| \F^W_t\big]\\
        =&\mE^Q[h(\Lambda|_{[0,t]\times U})|\F^W_t] \tilde g(W_{\cdot \wedge t })\\
        =&\mE^Q[h(\Lambda|_{[0,t]\times U})|\F^W_t] \mE^Q[g(W)|\F^{W}_t],
    \end{align}
    which verifies the desired conditional independence.
\end{proof}

\begin{lemma}\label{lemma:BMexistence}
    For any causal coupling $Q\in \cP_W(\cQ\times \cW)$, $W$ is a Brownian motion under $(Q,\mF^{\Lambda,W})$.
\end{lemma}
\begin{proof}
    In light of $W_\#Q=\mP_W$, the distributional properties of Brownian motion are satisfied. We only need to verify that, for any $s<t$, $W_t-W_s$ is independent of $\F^{\Lambda,W}_s$. To this end, for any bounded measurable functions $\varphi$ on $\mR^n$, $g$ on $\cW$ and $h$ on $\cQ$ such that $g$ only depends on $W_{\cdot \wedge s}$ and $h$ only depends on $\Lambda|_{[0,s]\times U}$, the conditional independence in the definition of causal coupling implies that
    \begin{align}
        \mE^Q[\varphi(W_t-W_s)h(\Lambda|_{[0,s]\times U})g(W_{\cdot \wedge s})]=& \mE^Q\big[g(W_{\cdot \wedge s})\mE^Q[\varphi(W_t-W_s)h(\Lambda|_{[0,s]\times U})|\F^W_s]\big]\\
        =&\mE^Q\big[g(W_{\cdot \wedge s})\mE^Q[\varphi(W_t-W_s)|\F^W_s]\mE^Q[h(\Lambda|_{[0,s]\times U})|\F^W_s]\big]\\
        =&\mE^Q[g(W_{\cdot \wedge s})h(\Lambda|_{[0,s]\times U})]\mE^Q[\varphi(W_t-W_s)].
    \end{align}
    By the arbitrariness of $g$ and $h$, we can use a monotone class extension argument to extend the equality above to any jointly measurable functions on  $W_{\cdot \wedge s}$ and $\Lambda|_{[0,s]\times U}$. As a consequence, $W_t-W_s$ is independent of $\F^{\Lambda,W}_s$.
\end{proof}

\section{Properties of rough martingale problems}\label{app:rough-martingale}

This appendix consists of useful properties of $\mP\in R(\bB,\bmu)$ (for a fixed input $(\bB,\bmu)$). The first several results will essentially be used to investigate $M^X$, and hence it is sufficient to take $\phi\in C_0^\infty(\mR^d)$. Recall that
\begin{align}
    &I_t(\phi) = \int_0^t \int_U \mL\phi(s,X_s,\mu_s,u)\Lambda_s(du)ds,\\
    &\cA_t(\phi) = \int_0^t \big(\Ti_s(\phi),\Ti'_s(\phi)\big)d \bB.
\end{align}

\begin{lemma}\label{lemma:roughIto:quadratic}
For any $\mP\in R(\bB,\bmu)$ and any $\phi\in C_0^\infty(\mR^d)$, we have $(2\cA_s(\phi)\Ti_s(\phi),2\Ti_s(\phi)^{\otimes2}+2\cA_s(\phi)\Ti'_s(\phi))\in \bD^{\beta,\beta'}_{\bB}L^{m,\infty}_{\mP}$, and the following result holds $\mP$-a.s.:
\begin{align}\label{roughIto:quadratic}
\cA_t(\phi)^2 = \int_0^t \big(2\cA_s(\phi)\Ti_s(\phi),2\Ti_s(\phi)^{\otimes2}+2\cA_s(\phi)\Ti'_s(\phi)\big)d\bB+\int_0^t |\Ti_s(\phi)|^2 ds,\quad \forall t\in [0,T].
\end{align}
\end{lemma}

\begin{proof}
    The conclusion is a slight generalization of the rough It\^o lemma established in Theorem 4.13 of \cite{RSDE}), since the quadratic function in our setting is not necessarily bounded.

    First of all, by the facts $(\cA(\phi),\Ti(\phi))\in \bD^{\alpha,\beta}_{\bB}L^{m,\infty}_{\mP}$ and $(\Ti(\phi),\Ti'(\phi))\in \bD^{\beta,\beta'}_{\bB}L^{m,\infty}_{\mP}$, it is not hard to show that $(2\cA_s(\phi)\Ti_s(\phi),2\Ti_s(\phi)^{\otimes2}+2\cA_s(\phi)\Ti'_s(\phi))\in \bD^{\beta,\beta'}_{\bB}L^{m,\infty}_{\mP}$ by using the identity
    \begin{align}
        &\cA_t(\phi)\Ti_t(\phi)-\cA_s(\phi)\Ti_s(\phi)-\Ti_s(\phi)^{\otimes 2}\delta B_{s,t}-\cA_s(\phi)\Ti'_s(\phi)\delta B_{s,t}\\
        =& \cA_s(\phi)(\Ti_t(\phi)-\Ti_s(\phi)-\Ti'_s(\phi)\delta B_{s,t})\\
        &+\Ti_s(\phi)(\cA_t(\phi)-\cA_s(\phi)-\Ti_s(\phi)\delta B_{s,t})\\
        &+\big(\Ti_t(\phi)-\Ti_s(\phi)\big)\big(\cA_t(\phi)-\cA_s(\phi)\big).
    \end{align}
    Therefore, the rough integration in \eqref{roughIto:quadratic} is well-defined.

    For any $s<t$, denote
\begin{align}
    &A_{s,t} = 2\cA_s(\phi)\Ti_s(\phi)\delta B_{s,t}+2\big(\Ti_s(\phi)^{\otimes 2}+\cA_s(\phi)\Ti'_s(\phi)\big)\mB_{s,t},\\
    &\tilde A_{s,t} = \cA_t(\phi)^2-\cA_s(\phi)^2-\int_s^t |\Ti_u(\phi)|^2 ds.
\end{align}
By the definition of rough integration and the uniqueness in stochastic sewing lemma, it is sufficient to prove that there exist $p_1>1/2$ and $p_2>1$ such that for any $s<t$,
\begin{align}
    &\|\tilde A_{s,t}-A_{s,t}\|_2\lesssim (t-s)^{p_1},\\
    &\|\mE_s[\tilde A_{s,t}-A_{s,t}]\|_2\lesssim (t-s)^{p_2}.
\end{align}
To this end, we first observe that
\begin{align}
    \tilde A_{s,t}-A_{s,t} =&\big(\cA_t(\phi)-\cA_s(\phi)\big)^2 - \bigg(\int_s^t |\Ti_u(\phi)|^2 du + 2\Ti_s(\phi)^{\otimes 2}\mB_{s,t}\bigg)\\
    &+2\cA_s(\phi)\big(\cA_t(\phi)-\cA_s(\phi)-\Ti_s(\phi)\delta B_{s,t}-\Ti'_s(\phi)\mB_{s,t}\big)\\
    =:&J_1+J_2.
\end{align}
When $m\geq 2$, with $m':=\frac{2m}{m-2}\in (2,\infty]$, we have $\bD_\bB^{\alpha,\beta}L^{m,\infty}_{\mP}\subset \bD_{\bB}^{\alpha,\beta}L^{2,m'}_{\mP}$. Then, H\"older's inequality yields that
\begin{align}
    \||J_2|\|_2\leq &2\|\cA_s(\phi)\|_m\big\| \|\cA_t(\phi)-\cA_s(\phi)-\Ti_s(\phi)\delta B_{s,t}-\Ti'_s(\phi)\mB_{s,t}|\F_s\|_2\big\|_{m'}\\
    \lesssim& \big\|\big(\cA_{\cdot}(\phi),\Ti(\phi)\big)\big\|_{\bB;\alpha,\beta;m,\infty}s^\alpha (t-s)^{\lambda_1}\\
    \lesssim& (t-s)^{\lambda_1},
\end{align}
for some $\lambda_1>1/2$. Similarly, it holds that
\begin{align}
    \|\mE^{\mP}[J_2|\F_s]\|_2\leq & \|\cA_s(\phi)\|_m\big\| \mE^{\mP}[\cA_t(\phi)-\cA_s(\phi)-\Ti_s(\phi)\delta B_{s,t}-\Ti'_s(\phi)\mB_{s,t}|\F_s]\big\|_{m'}\\
    \lesssim & (t-s)^{\lambda_2},
\end{align}
for some $\lambda_2>1$. Next, to deal with $J_1$, we recall by definition that $[\bB]_{s,t}=(\delta B_{s,t})^{\otimes 2}-(\mB_{s,t}+\mB^\mt_{s,t})=(t-s)I_{k\times k}$. Moreover, the notation of tensor operations gives $\Ti_s(\phi)^{\otimes2}\mB_{s,t} =\Ti_s(\phi)^{\otimes2}:\mB_{s,t}=\Ti_s(\phi)^{\otimes2}\mB_{s,t}^\mt$. As a result, for the second term in $J_1$, we have
\begin{align}
   \int_s^t |\Ti_u(\phi)|^2 du + 2\Ti_s(\phi)^{\otimes 2}\mB_{s,t}=&\int_s^t\big(|\Ti_u(\phi)|^2-|\Ti_s(\phi)|^2\big)d u + |\Ti_s(\phi)|^2 (t-s)  + 2\Ti_s(\phi)^{\otimes 2}\mB_{s,t}\\
   =&\int_s^t \big(|\Ti_u(\phi)|^2-|\Ti_s(\phi)|^2\big)d u+\Ti_s(\phi)^{\otimes 2}\big((\delta B_{s,t})^{\otimes 2}-(\mB_{s,t}+\mB^\mt_{s,t})\big)\\
   &+ \Ti_s(\phi)^{\otimes 2}(\mB_{s,t}+\mB_{s,t}^\mt)\\
   =& \int_s^t \big(|\Ti_u(\phi)|^2-|\Ti_s(\phi)|^2\big)d u + |\Ti_s(\phi)\delta B_{s,t}|^2.\label{J1:id}
\end{align}
On one hand, it holds that
\begin{align}
    \bigg\|\int_s^t \big(|\Ti_u(\phi)|^2-|\Ti_s(\phi)|^2\big)d u\bigg\|_2\lesssim & \int_s^t \|\Ti_u(\phi)-\Ti_s(\phi)\|_2 d u\lesssim (t-s)^{1+\beta}. \label{J1:id1}
\end{align}
On the other hand, we deduce that
\begin{align}
    \big(\cA_t(\phi)-\cA_s(\phi)\big)^2 -  |\Ti_s(\phi)\delta B_{s,t}|^2=&\big(\cA_t(\phi)-\cA_s(\phi)-\Ti_s(\phi)\delta B_{s,t}-\Ti'_s(\phi)\mB_{s,t}\big)^2\\
    &+2\Ti_s(\phi)\delta B_{s,t}\big(\cA_t(\phi)-\cA_s(\phi)-\Ti_s(\phi)\delta B_{s,t}-\Ti_s'(\phi)\mB_{s,t}\big)\\
    &+\Ti'_s(\phi)\mB_{s,t}\big(\cA_t(\phi)-\cA_s(\phi)+\Ti_s(\phi)\delta B_{s,t}\big)\\
    &=: J_{1,1}+J_{1,2}+J_{1,3}, \label{J11-J13}
\end{align}
which are estimated separately as follows: for some $\mu_1>1/2,\mu_2>1$,
\begin{align}
    &\|J_{1,1}\|_2  = \|\cA_t(\phi)-\cA_s(\phi)-\Ti_s(\phi)\delta B_{s,t}-\Ti'_s(\phi)\mB_{s,t}\|_4^2\lesssim (t-s)^{2\mu_1},\\
    &\|J_{1,2}\|_2 \lesssim (t-s)^\alpha \|\cA_t(\phi)-\cA_s(\phi)-\Ti_s(\phi)\delta B_{s,t}-\Ti_s'(\phi)\mB_{s,t}\|_2\lesssim (t-s)^{\mu_1+\alpha},\\
    &\|\mE^{\mP}[J_{1,2}|\F_s]\|_2\lesssim (t-s)^{\mu_2+\alpha} ,\\
    &\|J_{1,3}\|_2\lesssim (t-s)^{2\alpha}\big(\|\cA_t(\phi)-\cA_s(\phi)\|_2+(t-s)^\alpha\big)\lesssim (t-s)^{3\alpha}.\label{J1:id2}
\end{align}
As $3\alpha>1$, \eqref{J1:id}, \eqref{J1:id1}, \eqref{J11-J13} and \eqref{J1:id2} together give the desired estimation of $J_1$, and the proof is completed.
\end{proof}

Similarly to Lemma \ref{lemma:roughIto:quadratic}, the following multiplication formula also holds; its proof is omitted. A key observation is that $t\mapsto I_t(\phi)$ has better regularity than $t\mapsto \cA_t(\phi)$, and hence a similar argument applies.
\begin{lemma}\label{lemma:IA-multiplication}
    For any $\phi\in C_0^\infty(\mR^d)$, we have $(I_s(\phi)\Ti_s(\phi),I_s(\phi)\Ti'_s(\phi))\in \bD^{\beta,\beta'}_\bB L^{m,\infty}_{\mP}$, and the following result holds $\mP$-a.s.:
    \begin{align}
        I_t(\phi)\cA_t(\phi)=\int_0^t \cA_s(\phi)\mL\phi(s,X_s,\mu_s,u)\Lambda_s(du) d s+ \int_0^t \big(I_s(\phi)\Ti_s(\phi),I_s(\phi)\Ti'_s(\phi)\big)d \bB,  \forall t\in [0,T].
    \end{align}
\end{lemma}

\begin{lemma}\label{lemma:MA}
    For any $(\mP,\mF)$-martingale $\{M_t\}_{0\leq t\leq T}$ such that
    \begin{align}
    &\sup_{t\in [0,T]}\mE[M_t^2]dt<\infty, \label{multiplication:MA:integration}\\
   \text{and}\quad  &\big(M \Ti(\phi),M\Ti'(\phi)\big) \in \bD^{\beta,\beta'}_{\bB}L^{m,\infty}_{\mP}, \label{multiplication:MA:controlled}
    \end{align}
    $\{\tilde M_t\}_{0\leq t \leq T}$ is also a $(\mP,\mF)$-martingale, where
    \begin{align}\label{multiplication:MA}
        \tilde M_t:= M_t \cA_t(\phi) - \int_0^t \big(M_s \Ti_s(\phi),M_s\Ti'_s(\phi)\big) d \bB, \quad 0\leq t\leq T.
    \end{align}
\end{lemma}
\begin{proof}
    For fixed $0\leq s <t\leq T$, and a partition $\cP$ of $[0,t]$, let us define
    \begin{align}
\tilde M^\cP_t = M_t\sum_{[u,v]\in \cP}\big(\Ti_u(\phi)\delta B_{u,v}+\Ti'_u(\phi)\mB_{u,v}\big) - \sum_{[u,v]\in \cP}\big(M_u\Ti_u(\phi)\delta B_{u,v}+M_u\Ti'_u(\phi)\mB_{u,v}\big).
    \end{align}
Consider the partition of $[0,s]$ given by $\cP_s:=\{[u,v\wedge s]:[u,v]\in \cP\}$. We first note the facts that
\begin{align}
   & \mE^{\mP}\bigg[M_t \sum_{[u,v]\in \cP, u<s} \Ti_u(\phi)\delta B_{u,v}+ \Ti'_u(\phi)\mB_{u,v}\bigg|\F_s\bigg]\\
   =&M_s\bigg(\sum_{[u,v]\in \cP, u<s} \Ti_u(\phi)\delta B_{u,v}+ \Ti'_u(\phi)\mB_{u,v}\bigg),\label{ulesss}
\end{align}
and
\begin{align}
    &\mE^{\mP}\bigg[M_t \sum_{[u,v]\in \cP, u\geq s} \Ti_u(\phi)\delta B_{u,v}+ \Ti'_u(\phi)\mB_{u,v}\bigg|\F_s\bigg]\\=& \mE^{\mP}\bigg[\mE^{\mP}\bigg[M_t \sum_{[u,v]\in \cP, u\geq s} \Ti_u(\phi)\delta B_{u,v}+ \Ti'_u(\phi)\mB_{u,v}\bigg|\F_u\bigg]\bigg|\F_s\bigg]\\
    =&\mE^{\mP}\bigg[ \sum_{[u,v]\in \cP, u\geq s} M_u\Ti_u(\phi)\delta B_{u,v}+ M_u\Ti'_u(\phi)\mB_{u,v}\bigg|\F_s\bigg].\label{ugreaters}
\end{align}
Using \eqref{ulesss} and \eqref{ugreaters} leads to
\begin{align}
    \mE^{\mP}[\tilde M^\cP_t|\F_s] =& M_s\bigg(\sum_{[u,v]\in \cP, u<s} \Ti_u(\phi)\delta B_{u,v}+ \Ti'_u(\phi)\mB_{u,v}\bigg)- \sum_{[u,v]\in \cP, u<s} M_u\Ti_u(\phi)\delta B_{u,v}+ M_u\Ti'_u(\phi)\mB_{u,v}\\
    =&\tilde M_s^{\cP_s} + (M_s-M_{u_s})\big(\Ti_{u_s}(\phi)\delta B_{s,v_s}+ \Ti'_{u_s}(\phi)\mB_{s,v_s}\big)\\
    =:&\tilde M_s^{\cP_s} + R_s,\label{MA:martingale:P}
\end{align}
where $[u_s,v_s]\in \cP$ is chosen such that $u_s<s\leq v_s$. It follows from \eqref{multiplication:MA:integration} that
\begin{align}
    \|R_s\|_2\lesssim& \Big(2\sup_{t\in [0,T]}\|M_t\|_2\Big)(|\delta B_{s,v_s}|+|\mB_{s,v_s}|)\\
    \lesssim& (v_s-s)^\alpha\\
    \lesssim& |\cP|^\alpha.
\end{align}
By stochastic sewing lemma (Theorem 2.8 of \cite{RSDE}), it holds that
\begin{align}
    \mE\Big[\Big|\mE^{\mP}[\tilde M_t^\cP|\F_s]- \mE^{\mP}[\tilde M_t|\F_s]\Big|^2\Big]\leq \mE^{\mP}[|\tilde M^\cP_t-\tilde M_t|^2]\to 0,
\end{align}
as $|\cP|\to 0$. Sending $|\cP|\to 0$ on both sides of \eqref{MA:martingale:P}, we get
\begin{align}
    \mE[\tilde M_t|\F_s] = \tilde M_s.
\end{align}
Since $s<t$ was arbitrary, $\tilde M$ is a martingale.
\end{proof}

\begin{lemma}\label{lemma:Mphi:QV}
    For any $\mP\in R(\bB,\bmu)$ and any $\phi\in C_0^\infty(\mR^d)$, the following result holds $\mP$-a.s.:
    \begin{align}\label{Mphi:QV:eq}
    \langle M^X(\phi)\rangle _t= \int_0^t  |\nabla \phi(X_s)\sigma(s,X_s,\mu_s)|^2d s,\quad \forall t\in [0,T].
    \end{align}

\end{lemma}

\begin{proof}
    For any $\phi\in C_0^\infty(\mR^d)$, let us consider $\psi = \phi^2$. Direct computations lead to
    \begin{align}
        &\nabla \psi = 2\phi \nabla \phi,\\
        &\nabla^2\psi = 2\nabla\phi^{\otimes 2} + 2\phi\nabla^2\phi,
    \end{align}
    and it thus holds that
    \begin{align}
        &\mL\psi = 2\phi \mL\phi + |\nabla \phi \sigma|^2 + |\nabla \phi \sigma^0|^2,\\
        &\Ti(\psi) = 2\phi \Ti(\phi),\label{T-phi-psi}\\
        &\Ti'(\psi) = 2\phi \Ti'(\phi) + 2\Ti(\phi)^{\otimes2}. \label{Tp-phi-psi}
    \end{align}
    Recalling $\phi = M(\phi) + I(\phi) + \cA(\phi)$, we have
    \begin{align}
        M^X_t(\psi) =& \phi(X_t)^2 - 2\int_0^t\int_U\phi(X_s)\mL\phi(s,X_s,\mu_s,u)\Lambda_s(du)ds-\int_0^t|\nabla\phi(X_s)\hat\sigma^0_s|^2d s \\
        &-\int_0^t |\nabla \phi(X_s)\sigma(s,X_s,\mu_s)|^2ds\\
        &-\int_0^t \big(2\phi(X_s) \Ti_s(\phi), 2\phi(X_s)\Ti_s'(\phi)+2\Ti_s(\phi)^{\otimes 2}\big)d\bB,\\
        =&\phi(X_t)^2 -\int_0^t  |\nabla \phi(X_s)\sigma(s,X_s,\mu_s)|^2ds\\
        &-2\int_0^t\int_UM^X_s(\phi)\mL\phi(s,X_s,\mu_s,u)\Lambda_s(du)ds\\
        &-2\int_0^t\int_UI_s(\phi)\mL\phi(s,X_s,\mu_s,u)\Lambda_s(du)ds\\
        &-2 \int_0^t\int_U\cA_s(\phi)\mL\phi(s,X_s,\mu_s,u)\Lambda_s(du)ds\\
        &-\int_0^t \big(2\cA_s(\phi) \Ti_s(\phi), 2\cA_s(\phi)\Ti_s'(\phi)+2\Ti_s(\phi)^{\otimes 2}\big)d\bB-\int_0^t|\Ti_s(\phi)|^2d s\\
        &-\int_0^t \big(2I_s(\phi) \Ti_s(\phi), 2I_s(\phi)\Ti_s'(\phi)\big)d\bB-\int_0^t \big(2M^X_s(\phi) \Ti_s(\phi), 2M^X_s(\phi)\Ti_s'(\phi)\big)d\bB.
    \end{align}
    Moreover,
    \begin{align}
        M^X_t(\phi)^2 = \phi(X_t)^2 - 2M^X_t(\phi)\big(I_t(\phi)+\cA_t(\phi)\big) - \big(I_t(\phi) + \cA_t(\phi)\big)^2.
    \end{align}
    It therefore holds that
    \begin{align}
     & M^X_t(\phi)^2 -\int_0^t  |\nabla \phi(X_s)\sigma(s,X_s,\mu_s)|^2\Lambda_s(du)ds \\
     = & M^X_t(\psi) -\underbrace{2\bigg(M^X_t(\phi) I_t(\phi) - \int_0^t\int_UM^X_s(\phi)\mL\phi(s,X_s,\mu_s,u)\Lambda_s(du)ds\bigg)}_{(I)} \\
     & - \underbrace{2\bigg(M^X_t(\phi)\cA_t(\phi)-\int_0^t \big(2M^X_s(\phi) \Ti_s(\phi), 2M^X_s(\phi)\Ti_s'(\phi)\big)d\bB \bigg)}_{(II)} \\
     &- \underbrace{\bigg(I_t(\phi)^2 - 2\int_0^t\int_UI_s(\phi)\mL\phi(s,X_s,\mu_s,u)\Lambda_s(du)ds\bigg)}_{(III)}\\
     &- \underbrace{2\bigg(I_t(\phi)\cA_t(\phi) -  \int_0^t \cA_s(\phi)\mL\phi(s,X_s,\mu_s,u)\Lambda_s(du) d s - \int_0^t \big(I_s(\phi)\Ti_s(\phi),I_s(\phi)\Ti'_s(\phi)\big)d \bB \bigg)}_{(IV)} \\
     &- \underbrace{\bigg(\cA_t(\phi)^2 -\int_0^t \big(2\cA_s(\phi) \Ti_s(\phi), 2\cA_s(\phi)\Ti_s'(\phi)+2\Ti_s(\phi)^{\otimes 2}\big)d\bB- \int_0^t|\Ti_s(\phi)|^2d s\bigg)}_{(V)}. \label{eq:decomposition:M2}
    \end{align}
    Noting that $t\mapsto I_t(\phi)$ is absolutely continuous and applying It\^o's formula, we have that $(I)$ is a stochastic integral against $M(\phi)$ and  is thus a martingale. Next, if we verify that $M$ satisfies \eqref{multiplication:MA:integration}-\eqref{multiplication:MA:controlled}, then Lemma \ref{lemma:MA} implies that $(II)$ is also a martingale. Indeed, \eqref{multiplication:MA:integration} is a direct consequence by the boundedness of $\phi$ and its derivatives. To prove \eqref{multiplication:MA:controlled}, we first use the condition $\mP\in R(\bB,\bmu)$ to obtain $(\Ti(\psi),\Ti'(\psi))\in \bD^{\beta,\beta'}_{\bB}L^{m,\infty}_{\mP}$. It follows from \eqref{T-phi-psi}-\eqref{Tp-phi-psi} that
    \begin{align}
       & \Ti(\psi) = 2M(\phi)\Ti(\phi)+2I(\phi)\Ti(\phi)+2\cA(\phi)\Ti(\phi),\\
       & \Ti'(\psi) = \big(2\cA(\phi)\Ti'(\phi) +  2\Ti(\phi)^{\otimes 2}\big) + 2M(\phi)\Ti'(\phi)+2I(\phi)\Ti'(\phi).
    \end{align}
    By Lemmas \ref{lemma:roughIto:quadratic}-\ref{lemma:IA-multiplication}, $(\cA(\phi)\Ti(\phi), \cA(\phi)\Ti'(\phi) +  \Ti(\phi)^{\otimes 2})\in \bD^{\beta,\beta'}_{\bB}L^{m,\infty}_{\mP}$ and $(I(\phi)\Ti(\phi), I(\phi)\Ti'(\phi))\in \bD^{\beta,\beta'}_{\bB}L^{m,\infty}_{\mP}$. Therefore, by the definitions of relevant path norms, we get that $(M(\phi)\Ti(\phi), M(\phi)\Ti'(\phi))\in \bD^{\beta,\beta'}_{\bB}L^{m,\infty}_{\mP}$, i.e., \eqref{multiplication:MA:controlled}. Then, it holds that $(II)$ in \eqref{eq:decomposition:M2} is a martingale. By the fundamental theorem of calculus, and Lemmas \ref{lemma:roughIto:quadratic}-\ref{lemma:IA-multiplication}, (III)-(V) are all zeros. This completes the proof because by \eqref{eq:decomposition:M2},
    \begin{align}
        M^X_t(\phi)^2 -\int_0^t  |\nabla \phi(X_s)\sigma(s,X_s,\mu_s,u)|^2ds, \quad t\in [0,T],
    \end{align}
    is a martingale.
\end{proof}

\begin{remark}\label{rmk:cross-variation}
    In Lemmas \ref{roughIto:quadratic}--\ref{lemma:Mphi:QV} we study the martingale condition under the choice $\psi:=\phi^2$ and finally obtain the information of quadratic variation. However, with two possibly different $\phi_1,\phi_2\in C_0^\infty(\mR^d)$, as $M^X(\phi)$ is linear in $\phi$, the relation
    \begin{align}
        \langle M^X(\phi_1),M^X(\phi_2)\rangle=\frac{1}{4}\Big(\langle M^X(\phi_1)+M^X(\phi_2)\rangle-\langle M^X(\phi_1)-M^X(\phi_2)\rangle\Big)
    \end{align}
    directly gives the following generalization to \eqref{Mphi:QV:eq}:
    \begin{align}\label{MX:cross-variation}
        \langle M^X(\phi_1),M^X(\phi_2)\rangle_t=\int_0^t \nabla \phi_1(X_s)\sigma(s,X_s,\mu_s)\cdot \nabla \phi_2(X_s)\sigma(s,X_s,\mu_s) ds.
    \end{align}
    Moreover, choosing a vector valued $\vec \phi\in C_0^\infty(\mR^d;\mR^d)$ and considering $M^X(\vec\phi)$ in an element-wise manner (so that it is now a martingale valued in $\mR^d$), we derive from \eqref{MX:cross-variation} that
    \begin{align}
        \langle M^X(\vec\phi)\rangle_t=\int_0^t \nabla \vec\phi(X_s)\sigma\sigma^\mt(s,X_s,\mu_s)\big(\nabla \vec\phi(X_s)\big)^\mt  d s.
    \end{align}
\end{remark}

    Lemma \ref{lemma:Mphi:QV} characterizes the quadratic variation of $M(\phi)$ for any test function $\phi$. This will be the first step to establish the equivalence between martingale solutions and weak solutions (through the lens of Subsection 4.5, \cite{RSDE}) of RSDE in the relaxed control formulation. We next complete this characterization by also incorporating the cross variation between $X$ and $W$.

\begin{lemma}\label{lemma:cross-M-W}
    For any $\mP\in R(\bB,\bmu)$ and $ \phi^X\in C_0^\infty(\mR^d)$, $\phi^W\in C_0^\infty(\mR^n)$, the following holds $\mP$-a.s.:
    \begin{align}
        \langle M^X( \phi^X),M^W(\phi^W)\rangle_t=\int_0^t \nabla \phi^X(X_s)\sigma(s,X_s,\mu_s)\cdot \nabla \phi^W(W_s) ds,
    \end{align}
    where
    \begin{align}
        M^W(\psi)_t = \psi(W_t)-\frac{1}{2}\int_0^t\triangle \psi(W_s) d s =\int_0^t \nabla \psi(W_s) d W_s.
    \end{align}
\end{lemma}
\begin{proof}
    The proof is similar to that of Lemma \ref{lemma:Mphi:QV}. For completeness, we only provide a sketch here. Denote
    \begin{align}
        I^W_t(\phi^W)=\frac{1}{2}\int_0^t \triangle \phi^W(W_s) ds.
    \end{align}
    Recall that $\phi^X(X_t)=M^X_t(\phi^X)+I_t(\phi^X)+\cA_t(\phi^X)$, $\phi^W(W_t)=M^W_t(\phi^W)+I^W_T(\phi^W)$. The martingale condition \eqref{def:Mphiomega1-XW} applied to $\psi(x,w)=\phi^X(x)\phi^W(w)$ gives,
    \begin{align}
        M_t(\psi)=&\phi^X(X_t)\phi^W(W_t)-\int_0^t\int_U \phi^W(W_s)\mL\phi^X(s,X_s,\mu_s,u)\Lambda_s(du)\\
        &-\int_0^t \nabla \phi^X(X_s)\sigma(s,X_s,\mu_s)\cdot \nabla \phi^W(W_s)ds -\frac{1}{2}\int_0^t \phi^X(X_s)\triangle\phi^W(W_s)ds\\
        &-\int_0^t \big(\phi^W(W_s)\Ti_s(\phi^X),\phi^W(W_s)\Ti_s'(\phi^X)\big)d\bB, \quad t\in [0,T],
    \end{align}
    is a $(\mP,\mF)$ martingale. Consequently,
    \begin{align}
        &M^X_t(\phi^X)M^W_t(\phi^W)-\int_0^t \nabla \phi^X(X_s)\sigma(s,X_s,\mu_s)\cdot \nabla \phi^W(W_s)ds\\
        =&M_t(\psi) +\int_0^t \int_U M^W_s(\phi^W) \mL\phi(s,X_s,\mu_s,u)\Lambda_s(du)ds - I_t(\phi^X)M_t^W(\phi^W) \\
        &+\int_0^t \int_U I^W_s(\phi^W) \mL\phi(s,X_s,\mu_s,u)\Lambda_s(du)ds+\frac{1}{2}\int_0^t I_s(\phi^X)\triangle\phi^W(W_s)ds - I_t(\phi^X)I^W_t(\phi^W) \\
        &+\frac{1}{2}\int_0^t M^X_s(\phi^X)\triangle\phi^W(W_s)ds - M^X_t(\phi^X)I^W_t(\phi^W) \\
        &+\frac{1}{2}\int_0^t \cA_s(\phi^X)\triangle\phi^W(W_s)ds+\int_0^t\big(I^W_s(\phi^W)\Ti_s(\phi^X),I^W_s(\phi^W)\Ti'_s(\phi^X)\big)d\bB - \cA_t(\phi^X)I^W_t(\phi^W)\\
        &+\int_0^t\big(M^W_s(\phi^W)\Ti_s(\phi^X),M^W_s(\phi^W)\Ti'_s(\phi^X)\big)d\bB - \cA_t(\phi^X)M^W_t(\phi^W). \label{cross-X-W}
    \end{align}
    Using Lemmas \ref{lemma:IA-multiplication} and \ref{lemma:MA} and applying the classical It\^o formula to the five terms in \eqref{cross-X-W}, we conclude that the right-hand side is a martingale.
\end{proof}

\bibliographystyle{informs2014} 
\bibliography{ref.bib}    

\end{document}